\documentclass[a4paper,11pt]{amsart}

\usepackage[utf8]{inputenc}		
\usepackage[english]{babel}

\usepackage{amsfonts}			
\usepackage{amsmath}
\usepackage{amssymb}
\usepackage{amsthm}
\usepackage{mathtools}

\usepackage{soul}

\usepackage{esvect}				

\usepackage{hyperref}			
	\hypersetup{colorlinks=false, urlcolor=black, linkcolor=black}

\newcommand{\N}{\mathbb{N}}
\newcommand{\bN}{\mathbb{N}}						
\newcommand{\Z}{\mathbb{Z}}						
\newcommand{\Q}{\mathbb{Q}}						
\newcommand{\R}{\mathbb{R}}						
\newcommand{\bR}{\mathbb{R}}	

\newcommand{\bS}{\mathbb{S}}					
\newcommand{\bT}{\mathbb{T}}

\newcommand{\Hausd}{\mathcal{H}}

\DeclareMathOperator{\lov}{lov}
\DeclareMathOperator{\lodn}{lodn}
\DeclareMathOperator{\Dens}{Dens}
\newcommand{\curr}{\mathcal{D}}
\newcommand{\ncurr}{\mathbf{N}}
\DeclareMathOperator{\mass}{\mathbf{M}}

\newcommand{\eps}{\varepsilon}					

\newcommand{\dd}								
	{\mathop{}\!\mathrm{d}}						
\newcommand{\ddn}[1]							
	{\mathop{}\!\mathrm{d^{#1}}}

\newcommand{\abs}[1]							

\newcommand{\smallabs}[1]						
	{\lvert #1 \rvert}	
\newcommand{\norm}[1]							
	{\left\lVert #1 \right\rVert}	
\newcommand{\smallnorm}[1]						
	{\lVert #1 \rVert}						
\newcommand{\ip}[2]								
	{\left< #1 , #2 \right>}

\DeclareMathOperator{\id}{id}					
\DeclareMathOperator{\proj}{pr}					
\DeclareMathOperator{\vol}{vol}					
\DeclareMathOperator{\spt}{spt}					
\DeclareMathOperator{\tr}{tr}					
\DeclareMathOperator{\chain}{Chain}					

\DeclareMathOperator{\diam}{diam}

\newcommand{\loc}{\mathrm{loc}}

\newcommand{\cH}{\mathcal{H}}
\newcommand{\cM}{\mathcal{M}}
\newcommand{\cN}{\mathcal{N}}
\newcommand{\cB}{\mathcal{B}}

\newtheorem{thm}{Theorem}[section]{\bf}{\it}
\newtheorem{lemma}[thm]{Lemma}
\newtheorem{prop}[thm]{Proposition}
\newtheorem{cor}[thm]{Corollary}

{\bf}{\it}

\theoremstyle{definition}

\theoremstyle{remark}
\newtheorem{rem}[thm]{Remark}

\numberwithin{equation}{section}

\begin{document}

\title[Entropy in uniformly quasiregular dynamics]{Entropy in uniformly quasiregular dynamics}
\author[Kangasniemi]{Ilmari Kangasniemi}
\author[Okuyama]{Y\^{u}suke Okuyama}
\author[Pankka]{Pekka Pankka}
\author[Sahlsten]{Tuomas Sahlsten}
\address{Department of Mathematics and Statistics, P.O. Box 68 (Pietari Kalmin katu 5), FI-00014 University of Helsinki, Finland}
\email{ilmari.kangasniemi@helsinki.fi,pekka.pankka@helsinki.fi}
\address{
Division of Mathematics,
Kyoto Institute of Technology, Sakyo-ku,
Kyoto 606-8585 Japan}
\email{okuyama@kit.ac.jp}
\address{School of Mathematics, University of Manchester, UK}
\email{tuomas.sahlsten@manchester.ac.uk}

\begin{abstract}
Let $M$ be a closed, oriented, and connected Riemannian $n$-manifold, for $n\ge 2$, which is not a rational homology sphere. We show that, for a non-constant and non-injective uniformly quasiregular self-map $f\colon M\to M$, the topological entropy $h(f)$ is $\log \deg f$. This proves Shub's entropy conjecture in this case.
\end{abstract}

\thanks{I.K.~was supported by the doctoral program DOMAST of the University of Helsinki. Y.O.~was partially supported by the JSPS Grant-in-Aid for Scientific Research (C), 15K04924. P.P.~was partially supported by the Academy of Finland grant \#297258. T.S.~was partially supported by the ERC starting grant $\sharp$306494, Marie Sk{\l}odowska-Curie Individual Fellowship grant $\sharp$655310 and a start-up fund from the MIMS in the University of Manchester. 
}

\date{\today}

\maketitle

\section{Introduction}\label{sect:intro}

A well-studied problem in topological dynamics of continuous self-maps $f : M \to M$ on an $n$-manifold $M$ is to relate the \textit{topological entropy} $h(f)$ of $f$ to the spectrum of its induced linear map $f_* : H_*(M;\R) \to H_*(M;\R)$ in homology, see for example the survey of Katok \cite{Katok-1977} for definitions and history of this problem. Shub conjectured \cite[\S V]{Shub-BAMS-1974} that the topological entropy $h(f)$ is bounded from below by $\log s(f_*)$, where $s(f_*)$ is the spectral radius of the action of $f$ to the homology of $M$.  The conjecture was proved for holomorphic maps $f\colon \mathbb{C}P^m \to \mathbb{C}P^m$ by Gromov in a preprint \cite{Gromov-2003} from 1977 and for $C^\infty$-smooth maps by Yomdin \cite{Yomdin-1987} in 1987.

One direction in Gromov's argument \cite{Gromov-2003} is based on a general result of Misiurewicz and Przytycki \cite{Misiurewicz-Przytycki-1977} that, 
for a $C^1$-smooth self-map $f\colon M \to M$ of a closed and oriented Riemannian manifold $M$, the logarithm of the degree $\log |\deg f|$ is a lower bound for the topological entropy. The continuity of the derivative $D f$ of the map $f$ plays a crucial role in the proof of Misiurewicz and Przytycki, which is based on the use of a continuous cochain $x\mapsto J_f(x)$ given by the Jacobian $J_f$ of the map $f$. The continuity of the derivative plays the same crucial role in the method of Yomdin \cite{Yomdin-1987}, which is based on real-algebraic sets. 

It is known that the smoothness assumptions on the map may be relaxed by additional topological assumptions on the space $M$. For example, Misiurewicz and Przytycki proved in \cite{Misiurewicz-Przytycki-1977} the entropy conjecture for all continuous maps $f : \bT^n \to \bT^n$. 

In this paper we consider the entropy conjecture in the quasiconformal category. The mappings we consider are not $C^1$-smooth but merely Sobolev regular. The distortion assumption given by quasiconformality conditions together with methods from geometric measure theory allow us to deal with the complications caused by the lack of pointwise differentiability. Before stating the main theorem, we introduce the class of uniformly quasiregular maps.

A continuous map $f\colon M\to N$ between oriented Riemannian $n$-manifolds $M$ and $N$, $n \geq 2$, is \emph{$K$-quasiregular for $K\ge 1$} if $f$ belongs to the Sobolev space $W^{1,n}_\loc(M,N)$ and satisfies the distortion inequality
\begin{equation}
\label{eq:bd}
\norm{Df(x)}^n \le K J_f(x) \quad \text{for Lebesgue a.e. }x\in M;
\end{equation}
here $\norm{Df}$ is the operator norm of the differential $Df$ of $f$ and $J_f$ is the Jacobian determinant $J_f = \det Df$, that is, $J_f \vol_M = f^*\vol_M$. In this terminology, \emph{quasiconformal maps} are quasiregular homeomorphisms, 
and $1$-quasiregular maps between Riemann surfaces are holomorphic; see e.g.~Rickman \cite[Section I.2]{Rickman} and references therein. As a technical point, we mention that by a theorem of Reshetnyak, a quasiregular map is either a discrete and open map or constant. Note also that the degree of a non-constant quasiregular map between closed and oriented Riemannian manifolds is positive.

A quasiregular self-map $f\colon M\to M$ is \emph{uniformly $K$-quasiregular} if all of its iterates $f^k = f^{\circ k} = f \circ \cdots \circ f$ for $k\ge 1$ are $K$-quasiregular. Uniformly quasiregular maps admit rich dynamics akin to dynamics of holomorphic maps of one complex variable. We refer to a survey of Martin \cite{Martin2014} for a detailed account on uniformly quasiregular maps, and merely mention here that a uniformly quasiregular map $f\colon M\to M$ induces a measurable conformal structure on $M$ in which the mapping $f$ could be considered as a rational map of $M$.

Our main theorem reads as follows; recall that an $n$-manifold $M$ is a \emph{rational cohomology sphere} if $H^*(M;\R)$ is isomorphic to $H^*(\bS^n;\R)$.

\begin{thm}\label{thm:main_result}
Let $f \colon M \to M$ be a uniformly quasiregular self-map of degree at least $2$ on a closed, connected, and oriented Riemannian $n$-manifold $M$ which is not a rational cohomology sphere. Then 
\[
	h(f) = \log \deg f.
\]
\end{thm}
It follows from \cite{KangasniemiPankka} that $s(f_*)=\deg f$ for non-constant uniformly quasiregular self-maps $f\colon M\to M$. Theorem \ref{thm:main_result} therefore yields the equality
\[
h(f) = \log s(f_*)
\]
answering to the Shub's entropy conjecture to the positive in this case. Note that, for expanding uniformly quasiregular mappings, Shub's entropy conjecture follows from results of Ha\"issinsky and Pilgrim \cite[Theorems 3.5.6 and 4.4.4]{Haissinsky-Pilgrim_CoaConfDYn}. 

In the proof of Theorem \ref{thm:main_result} we obtain estimates $h(f)\ge \log \deg f$ and $h(f)\le \log \deg f$ for the entropy by different methods. The lower bound employs Lyubich's variational method \cite{Lyubich-1983} and the properties \cite{Kangasniemi, OkuyamaPankka} of the equilibrium measure $\mu_f$ associated $f$. The upper bound is related to \cite[(5.0)]{Gromov-2003} in Gromov's article and it follows from isoperimetric arguments for Federer-Fleming currents \cite{Federer}. As we will discuss shortly, the cohomological assumption on $M$ has no role in the proof of the upper bound. It remains an open question whether the lower bound $h(f)\ge \log \deg f$ holds also for uniformly quasiregular mappings on rational cohomology spheres.

\medskip

In order to obtain the lower bound $h(f)\ge \log \deg f$, the main obstacle is the lack of continuity of the derivative $D f$. For this, we use the $f$-balanced measure $\mu_f$ from \cite{OkuyamaPankka} and the integer valued cochain $x\mapsto i(x,f)$ given by the local index of the map $f$ in place of cochain $x\mapsto J_f(x)$ which is only measurable in this setting.

By \cite[Theorem 1.2]{Kangasniemi}, the cohomological assumption on the manifold $M$ yields that the measure $\mu_f$ is absolutely continuous with respect to the Lebesgue measure of $M$. Using this fact, we show that the measure $\mu_f$ satisfies $h_{\mu_f}(f) = \log \deg f$, where $h_{\mu_f}(f)$ is the measure theoretic entropy of $f$ with respect to the measure $\mu_f$. The variational principle of the entropy now yields the required lower bound $h(f) \ge h_{\mu_f}(f) = \log \deg(f)$. We thank Peter Ha\"issinsky for pointing out a simplified version of the original proof based on measure theoretic Jacobians.

We also note that, as a consequence of the method of proof, we also obtain the following observation.
\begin{cor}
	Let $f \colon M \to M$ be a uniformly quasiregular self-map of degree at least $2$ on a closed, connected, and oriented Riemannian $n$-manifold $M$ which is not a rational cohomology sphere. Then the measure $\mu_f$ from \cite{OkuyamaPankka} is a measure of maximal entropy.
\end{cor}

Moreover, we note that the absolute continuity of $\mu_f$ is only used in the proof to obtain that the branch set of $f$ has zero measure in $\mu_f$. Hence, the proof in fact also gives us the following, more general version of Theorem \ref{thm:main_result}.

\begin{thm}\label{thm:main_result_v2}
	Let $f \colon M \to M$ be a uniformly quasiregular self-map of degree at least $2$ on a closed, connected, and oriented Riemannian $n$-manifold $M$. Suppose that there exists an $f$-balanced Borel probability measure $\mu$ on $M$. Then the measure-theoretic entropy of $\mu$ satisfies
	\begin{equation}\label{eq:meas_entr_i_version}
		h_{\mu}(f) \geq \log \deg f - \int_M \log i(x, f) \dd \mu(x),
	\end{equation}
	where $i(\cdot, f)$ denotes the local index of $f$. In particular, if the branch set $B_f$ of $f$ satisfies $\mu(B_f) = 0$, then
	\[
		h(f) = \log \deg f,
	\]
	and $\mu$ is a measure of maximal entropy for $f$.
\end{thm}

For an estimate similar to \eqref{eq:meas_entr_i_version} in the setting of non-Archimedean dynamics, see Favre--Rivera-Letelier \cite[Section 4]{Favre-RivieraLetelier_ArchmiedeanDyn}.

\medskip

The upper bound $h(f) \le \log \deg f$ follows from the inequality
\begin{equation}
\label{eq:Gromov}
h(f) \le \log \deg f + n \log K
\end{equation}
for $K$-quasiregular self-maps $f\colon M \to M$; see \cite[(5.0)]{Gromov-2003} and the ensuing isopetrimeric argument on how to prove it. Since it seems to have gone unnoticed in the literature that the isoperimetric argument in \cite{Gromov-2003} yields a more general result, we discuss the proof of \eqref{eq:Gromov} in detail using the language of Federer-Fleming theory of currents. In the heart of the proof of \eqref{eq:Gromov} is the following uniform Ahlfors regularity result for graphs of maps, whose components are quasiregular.

\begin{thm}\label{thm:Gromov_interpretation}
Let $M$ and $N$ be closed, connected, and oriented
Riemannian $n$-manifolds for $n \geq 2$, $K\ge 1$, and 
let $g = (f_1, \ldots, f_k) \colon M \to N^k$ be a map from $M$ to $N^k$, $k\in\bN$,
where $f_1, \ldots, f_k$ are non-constant $K$-quasiregular maps $M\to N$. 
Then the image $\Gamma=\Gamma_g:=g(M)$ is Ahlfors $n$-regular. More precisely, there exists a constant $C>0$ depending only on $n,M,N$ and $f_1$ with the property that, for $y \in \Gamma$ and $r\in(0,\diam\Gamma]$, we have
\begin{gather*}
 \frac{1}{C k^\frac{n^2}{2} K^{n-1}(\min_j \deg f_j)^n}\le
 \frac{\cH^n(B_{\Gamma}(y, r))}{r^n}
 \leq C k^\frac{n}{2} K\max_j \deg f_j,
\end{gather*}
where $B_{\Gamma}(y, r) = \Gamma \cap B_{N^k}(y,r)$ with distance in $N^k$ induced by the product Riemannian metric. 
\end{thm}

Using this theorem we prove inequality \eqref{eq:Gromov} in Section \ref{sec:Gromov-proof}; see Theorem \ref{thm:Gromov-original}. This completes the proof of Theorem \ref{thm:main_result}.

The proof of Theorem \ref{thm:Gromov_interpretation} consists of two parts.
The upper estimate for the Hausdorff measure reduces to the area formula for Sobolev mappings. The lower estimate is more delicate. Since the mapping $g$ is merely Sobolev regular, we consider an $n$-current associated to $\Gamma_g$. The key step in the proof is to apply slicing and an isoperimetric inequality to this $n$-current to obtain a local lower bound for the volume of $\Gamma_g$. It seems to us that this is also the idea in the proof of \cite[(5.0)]{Gromov-2003}, although it does not use currents explicitly.

\medskip

We finish this introduction with a discussion on the relation of our results to open questions on uniformly quasiregular dynamics. In the case of Riemann surfaces, holomorphic dynamics has a clear trichotomy into different cases: the sphere $\bS^2$ carries a rich theory with various examples, on the torus $\bT^2$ the mappings are so-called Latt\`es maps, and on higher dimensional surfaces the theory collapses to dynamics of homeomorphisms. 

On higher-dimensional Riemannian manifolds, a similar trichotomy seems to arise in uniformly quasiregular dynamics. The sphere $\bS^n$ and other spherical space forms admit a rich theory, see e.g.\ Iwaniec--Martin \cite{Iwaniec-Martin_AASF}, Peltonen \cite{Peltonen-CGD}, and Martin--Peltonen \cite{Martin-Peltonen-PAMS}. The torus $\bT^n$ and its branched quotients admit uniformly quasiregular maps of Latt\`es type, see e.g.\ Mayer \cite{Mayer1997paper} and Martin--Mayer--Peltonen \cite{Martin-Mayer-Peltonen}.
Finally, the existence of a uniformly quasiregular map $M \to M$ on a closed manifold yields that the manifold $M$ is so-called quasiregularly elliptic, that is, there exists a non-constant quasiregular map $\R^n \to M$; see Kangaslampi \cite{Kangaslampi-thesis} or Iwaniec--Martin \cite[Theorem 19.9.3]{Iwaniec-Martin-book}. 
Thus, hyperbolic Riemannian manifolds and manifolds with large fundamental group or cohomology do not carry uniformly quasiregular maps by results of Varopoulos \cite[Theorem X.11]{Varopoulos-book} and Bonk--Heinonen \cite{Bonk-Heinonen_Acta}. More precisely, the dimension of the cohomology ring $H^*(M;\R)$ of $M$ is at most $2^n$ by the main theorem of \cite{Kangasniemi}; see also Prywes \cite{Prywes}.

To complete this picture, it becomes a question whether a general quasiregularly elliptic manifold carries a uniformly quasiregular mapping of higher degree, and whether these mappings are actually Latt\`es maps if the manifold in question is not a rational cohomology sphere. Encouraged by results and conjectures of Martin and Mayer in \cite{Martin-Mayer-2003} on uniformly quasiregular self-maps of spheres, we expect the second question to have a positive answer. The following conjecture is from \cite{Kangasniemi}:
\emph{Let $M$ be a closed, oriented, and connected Riemannian $n$-manifold for $n\ge 2$ which is not a rational cohomology sphere. Then every uniformly quasiregular self-map $f$ of $M$ comes from the Latt\`es construction.}

We find the question interesting since, as pointed out in Martin--Mayer \cite{Martin-Mayer-2003}, it is similar to the invariant line field conjecture of Man\'e, Sad, and Sullivan \cite{Mane-Sad-Sullivan}.

\subsection*{Organization of the article}

The article consists of two parts; Section \ref{sec:qrmaps} discussing the preliminaries on quasiregular maps is common to both of these. In the first part (Sections \ref{sec:pre_ent}--\ref{sec:lower_bound}), we prove the lower bound $h(f) \ge \log \deg f$ for the topological entropy using Lyubich's method based on measure theoretic entropy. 

In the second part (Sections \ref{sect:Federer_Fleming}--\ref{sec:Gromov-proof}) we recall first some results in the Federer--Fleming theory of currents in Section \ref{sect:Federer_Fleming}.  In Sections \ref{sec:Ahlfors_euclidean} and \ref{sec:Gromov_for_maps}, we then discuss the proof of Theorem \ref{thm:Gromov_interpretation} based on Gromov's original argument. Finally, in Section \ref{sec:Gromov-proof}, we show how the upper bound $h(f)\le \log \deg f$ follows from Theorem \ref{thm:Gromov_interpretation}. 

\bigskip
\noindent
{\bf Acknowledgments} We thank Petri Ola for suggesting us to look at the Gr\"onwall's inequality, which plays a key role in the upper bound for the entropy. We also thank Peter Ha\"issinsky, who in the process of pre-examining the PhD thesis of the first named author suggested multiple improvements to the paper.



\section{Preliminaries on quasiregular maps} \label{sec:qrmaps}

\subsection{Quasiregular maps}
Let $n\ge 2$, and let $M$ and $N$ be oriented Riemannian $n$-manifolds. 
By a theorem of Reshetnyak, a non-constant quasiregular map $f:M\to N$ is open and discrete, that is, $f(W)\subset N$ is open for any open set $W \subset M$ and $f^{-1}\{y\}\subset M$ is discrete for every $y\in N$. Moreover, $f$ satisfies the Lusin (N)-condition, that is,  $f(E)\subset N$ is Lebesgue null if $E\subset M$ is a null set. The \emph{branch set} $B_f$ of $f$ is the set of points at which $f$ fails to be a local homeomorphism. The branch set $B_f$ has topological dimension at most $n-2$ by the Cernavskii--V\"ais\"al\"a theorem (see \cite{Vaisala1966paper}) and Lebesgue measure zero.

For $E \subset M$ and $y \in N$, the \emph{multiplicity $N(f, y, E)$ of $f$ at $y$ with respect to $E$} is $\#(f^{-1}\{y\} \cap A)$. We set also $N(f, y):= N(f, y, M)$, $N(f, E):= \sup_{y \in N} N(f, y, E)$, and 
\begin{gather*}
 N(f):= \sup_{y \in N} N(f, y) = N(f, N).
\end{gather*}

As a preliminary step for the definition of the local index of $f$ at $x$, we denote by $B_N(y,r)$ the metric ball of radius $r>0$ centered at $y\in N$ in $N$. Since $f$ is discrete and open, there exists, for each $x\in M$, a radius $r_x>0$ for which the $x$-component $U(x,f,r_x)$ of the preimage $f^{-1}B_N(f(x),r_x)$ is a normal neighborhood of $x$, that is, we have $fU(x,f,r_x) = B_N(f(x),r_x)$, $\partial fU(x,f,r_x) = \partial B_N(f(x),r_x)$, and $f^{-1}(f(x)) \cap \overline{U(x,f,r_r)} = \{x\}$. In particular, $f$ restricts to a proper map 
\[
f|_{U(x,f,r_x)} \colon U(x,f,r_x) \to B_N(f(x), r_x)
\]
and induces a homomorphism 
\[
(f|_{U(x,f,r_x)})^* \colon H^n_c(B_N(x,r);\Z) \to H^n_c(U(x,f,r_x);\Z)
\]
in compactly supported cohomology.

The \emph{local index $i(x,f)\in \Z$ of $f$ at $x$} is the unique integer satisfying
\[
(f|_{U(x,f,r_x)})^*c_{B_N(f(x),r_x)} = i(x,f) c_{U(x,f,r_x)},
\]
where the cohomology classes $c_{U(x,f,r_x)}$ and $c_{B_N(f(x),r_x)}$ are generators of $H^n_c(B_N(x,r);\Z)$ and $H^n_c(U(x,f,r_x);\Z)$, respectively, induced by orientations of $M$ and $N$. The local index is independent on $r_x$ and hence well-defined. Note that, if $f$ is non-constant, we have $i(x,f)\ge 1$ for each $x\in M$ and we have the characterization that $x\in B_f$ if and only if $i(x,f)>1$. 

More globally, for a quasiregular map $f\colon M\to N$ between closed, oriented, and connected Riemannian $n$-manifolds $M$ and $N$, the degree $\deg f\in \Z$ of $f$ is the integer satisfying $f^*(c_N) = (\deg f)c_M$ for generators $c_M$ and $c_N$ of $H^n(M;\Z)$ and $H^n(N;\Z)$, respectively. Again, if $f$ is non-constant, then $\deg f\ge 1$ and
\begin{gather*}
\sum_{x \in f^{-1}\{y\}} i(x,f)=\deg f\quad\text{for every }y\in N.
\end{gather*} 
In particular, we have $N(y, f) = N(f) = \deg f$ for every $y \in N \setminus f(B_f)$.

We refer to the monograph of Rickman \cite[Chapter I]{Rickman} for a more detailed discussion on these properties of quasiregular mappings.

\subsection{Uniformly quasiregular self-maps}

Let $f\colon M\to M$ be a uniformly quasiregular self-map of a closed, oriented, and connected Riemannian $n$-manifold $M$. The \emph{Fatou set} $F(f)$ of $f$ is the region of normality of the family $\{f^k:k\in\bN\}$, that is, the set of all points $x \in M$ for which $\{f^k|U : k \in \bN\}$ is normal 
on some open neighborhood $U$ of $x$. 
The \emph{Julia set $J(f)$} of $f$ is $M\setminus F(f)$.

The Julia set $J(f)$ is non-empty if $\deg f>1$. In this case, there exists by \cite{OkuyamaPankka} an $f$-balanced probability measure $\mu_f$ on $M$, that is, 
\[
	f^*\mu_f = (\deg f)\mu_f.
\]
Here, the pull-back measure is defined using the push-forward of continuous functions under quasiregular maps; see Heinonen--Kilpel\"ainen--Martio \cite[Section 14]{HeinonenKilpelainenMartio2006book}. In particular, if $\eta \in C(M, \R)$ is a continuous function on the closed manifold $M$, then the formula
\begin{equation}\label{eq:pushforward_def}
	(f_* \eta)(x) = \sum_{z \in f^{-1}\{x\}} i(z, f) \eta(z)
\end{equation}
for $x \in M$ defines a continuous function $f_* \eta \in C(M, \R)$. Hence, given a finite Borel measure $\mu$ on $M$, the Riesz representation theorem provides a unique regular Borel measure $f^* \mu$ satisfying
\[
	\int_M \eta \dd f^*\mu = \int_M f_*\eta \dd \mu
\]
for every $\eta \in C(M, \R)$.

The measure $\mu_f$ is the weak-$\ast$ -limit of the measures $(\deg f^k)^{-1} (f^k)^*\vol_M$, where we identify the volume form $\vol_M$ with the Lebesgue measure on $M$ and tacitly assume that $\vol_M(M) = 1$, and the support of $\mu_f$ is the Julia set $J(f)$ of $f$. From now on, we use the notation $\mu_f$ to denote this particular measure.

By \cite[Theorem 1.2]{Kangasniemi}, the measure $\mu_f$ is absolutely continuous with respect to the Lebesgue measure if the manifold $M$ is not a rational cohomology sphere. Thus, similarly as in the holomorphic dynamics of one complex variable, we have that the branch set has $\mu_f$-measure zero. We record this fact as a lemma for the further use.

\begin{lemma}\label{lma:nobranchset} 
Let $M$ be a closed, oriented, and connected Riemannian $n$-manifold for which $H^n(M;\Q)\not\cong H^*(\bS^n;\Q)$, and let $f\colon M\to M$ be a uniformly quasiregular self-map of degree at least $2$. Then 
\[
\mu_f(f^{-1} f(B_f)) = \mu_f(f(B_f)) = \mu_f(B_f) = 0.
\]
\end{lemma}
\begin{proof}
By Rickman \cite[Proposition I.4.14]{Rickman} and an application of bilipschitz charts, the sets $f^{-1} f(B_f)$, $f(B_f)$, and $B_f$ are Lebesgue null. Since $\mu_f$ is absolutely continuous with respect to Lebesgue measure by \cite[Theorem 1.2]{Kangasniemi} under the assumption $H^*(M;\Q)\not\cong H^*(\bS^n;\Q)$, the claim follows.
\end{proof}

Finally, we point out an explicit formula for the measures of Borel sets under a pulled-back measure $f^*\mu$. We first note that we may in fact define $f_* \eta$ even for non-continuous $\eta$ using \eqref{eq:pushforward_def}.

\begin{lemma}\label{lma:pushfwd_set_measure}
	Let $M$ be a closed, oriented, and connected Riemannian $n$-manifold, let $f\colon M\to M$ be a non-constant quasiregular self-map, and let $\mu$ be a finite Borel measure on $M$. Then for every Borel set $E \subset M$, the function $f_* \mathcal{X}_E$ is Borel, and moreover we have
	\[
		f^* \mu(E) = \int_M f_* \mathcal{X}_E \dd \mu,
	\] 
	where $\mathcal{X}_E$ denotes the characteristic function of $E$.
\end{lemma}
\begin{proof}
	The most involved part of the proof is showing that $f_* \mathcal{X}_E$ is Borel; after that, the rest is a standard measure theory argument. Indeed, suppose that $f_* \mathcal{X}_E$ is Borel for every Borel set $E$. Then we may define a Borel measure $\nu$ on $M$ by
	\[
		\nu(E) = \int_M f_* \mathcal{X}_E \dd \mu.
	\]
	It is easily seen that $\nu$ is a finite Borel measure. Hence, it is also regular. Moreover, $\nu$ satisfies by definition the formula 
	\begin{equation}\label{eq:nu_formula}
		\int_M \eta \dd \nu = \int_M f_* \eta \dd \mu
	\end{equation}
	for simple Borel functions $\eta$. Since the operator $f_*$ is bounded in the sup-norm, we obtain \eqref{eq:nu_formula} for continuous $\eta$ by approximating with simple functions. Hence, $\nu = f^* \mu$ by the uniqueness of the measure given by the Riesz representation theorem, and the remainder of the claim holds.
	
	It remains therefore to show that $f_* \mathcal{X}_E$ is Borel whenever $E$ is Borel. We partition $M$ into sets $A_1, \dots, A_{\deg f}$, where $i(x, f) = j$ whenever $x \in A_j$. The sets $A_j$ are Borel, since the map $i(\cdot, f)$ is upper semicontinuous; see eg.\ \cite[Proposition I.4.10]{Rickman} for the argument.
	
	We begin by showing a special case. Suppose $E \subset M$ is Borel, $f\vert_E$ is injective, and $E \subset A_j$ for some $j \in \{1, \dots, \deg f\}$. Then it is reasonably easily seen from the definition of the push-forward that $f_* \mathcal{X}_E = j \mathcal{X}_{f(E)}$. Since $f$ is a continuous finite-to-one map, it maps Borel sets to Borel sets: see e.g.\ \cite[Theorem 4.12.4]{Srivastava_Borel}. Therefore, in this case $f_* \mathcal{X}_E$ is Borel.
	
	Next, we find a partition of $M$ into countably many Borel sets $B_i \subset M$, such that $f\vert_{B_i}$ is injective and $B_i \subset A_j$ for some $j$. Indeed, for a given $j \in \{1, \dots, \deg f\}$ if $x \in A_j$ and $U$ is a normal neighborhood of $x$ with respect to $f$, then $f$ is injective on $A_j \cap U$. Since $A_j$ is a subset of a second-countable metric space $M$, we may cover it with countably many such sets $A_j \cap U$. Hence, we obtain the desired Borel partition $B_j$.
	
	Finally, suppose that $E \subset M$ is Borel. Then we may write $\mathcal{X}_E$ as a countable sum of functions $\mathcal{X}_{E \cap B_j}$. By the special case we covered, $f_* \mathcal{X}_{E \cap B_j}$ is Borel for every $j$. Hence, we may write $f_* \mathcal{X}_E$ as a countable sum of non-negative Borel functions. Since pointwise limits of Borel functions are Borel, we obtain that $f_* \mathcal{X}_E$ Borel, which concludes the proof.
\end{proof}

\section{Preliminaries on entropy}
\label{sec:pre_ent}

\subsection{Topological entropy} \label{sec:topentropy}
Let $(X, d)$ be a metric space. For each $k\in\bN$, we denote by $d_{k,\infty}$ the sup-metric, induced by $d$, on $X^{k+1}$. That is, for any $x=(x_0,\ldots,x_k)$ and $x'=(x_0',\ldots,x_k')\in X^{k+1}$,
\[
d_{k,\infty}(x, x') := \sup_{j\in\{0,\ldots,k\}} d(x_j, x'_j).
\]
For any $\eps>0$ and $Y \subset X^{k+1}$, we also define the counting function 
\[
N_\eps(Y) := \max \Bigl\{\# E : E \subset Y,
\inf_{x,x'\in E, x\ne x'}d_{k,\infty}(x, x') \geq \eps\Bigr\}
\]
for the discrete volume of $Y$ at scale $\varepsilon$.

A \emph{graph over $X$} is by definition a subset of $X^2$. 
For any $\Gamma\subset X^2$,
the \emph{$k$-chain of $\Gamma$} is defined by
\begin{align*}
 \chain_k(\Gamma)
 := \{(x_0, \ldots, x_k) \in X^{k+1} : (x_{j-1}, x_j) \in \Gamma \text{ for any } j \in \{1, \ldots, k\} \},
\end{align*}
and for each $\epsilon>0$, we set
\[
 h_\eps(\Gamma) := \limsup_{k \to \infty} \frac{1}{k} \log(N_\eps (\chain_k(\Gamma))).
\]
The \emph{entropy} $h(\Gamma)$ of $\Gamma$ is 
\[
h(\Gamma) := \lim_{\eps \to 0} h_\eps(\Gamma);
\]
note that the limit on the right hand side always exists.

The Bowen--Dinaburg definition of the {\em topological entropy}
$h(f)$ of a continuous self-map $f$ on $X$ is
\begin{gather*}
 h(f):=h(\Gamma_{(\id_X,f)}),
\end{gather*}
where $\Gamma_{(\id_X,f)}:=(\id_X,f)(X)\subset X^2 $ is the graph of $f$. The topological entropy is a topological invariant whenever $(X,d)$ is compact \cite{Bowen1971}.

\subsubsection{Entropy, volume, and density}
Let $M$ be a closed Riemannian $n$-manifold. For each $k\in \N$, we let $\cH^n$ to be the Hausdorff $n$-measure on the $(nk)$-dimensional product Riemannian manifold $M^k$.

For each $\eps > 0$, the \emph{$\eps$-density} $\Dens_\eps(Y)$ of a $\Hausd^n$-measurable set $Y \subset M^{k+1}$ is defined by
\begin{align*}
	\Dens_\eps(Y) 
	= \inf_{x \in Y} \Hausd^n(Y \cap D_{k,\infty}(x, \eps)),
\end{align*}
where $D_{k,\infty}(x,\eps):=\{y\in M^{k+1}:d_{k,\infty}(x,y)<\epsilon\}$.

For any $\Gamma\subset M^2$,
the \emph{logarithmic volume} $\lov(\Gamma)$ of $\Gamma$ is defined by
\[
\lov(\Gamma) = \limsup_{k \to \infty} \frac{1}{k}\log\bigl(\Hausd^n(\chain_k(\Gamma))\bigr),
\]
and the \emph{logarithmic density} $\lodn(\Gamma)$ of $\Gamma$ by
\[
\lodn(\Gamma) = \limsup_{\eps \to 0} \lodn_\eps(\Gamma),
\]
where, for each $\epsilon>0$, 
\[
\lodn_\eps(\Gamma) := \liminf_{k \to \infty} \frac{1}{k} \log\bigl(\Dens_\eps (\chain_k(\Gamma))\bigr).
\]

For completeness, we include a proof of the following key estimate.
\begin{thm}[{\cite[(1.1)]{Gromov-2003}}]\label{lem:lov-lodn-bound}
Let $M$ be a closed Riemannian $n$-manifold and let $\Gamma\subset M^2$ be a graph. Then
\begin{equation}
\label{eq:h-lov-lodn}
h(\Gamma) \leq \lov(\Gamma) - \lodn(\Gamma).
\end{equation}
\end{thm}
\begin{proof}
Let $k \geq 2$, $\eps > 0$, and $\delta > 0$,
and let $d$ be the induced Riemannian distance in $M$ 
and $d_{k,\infty}$ be the sup-metric on $M^{k+1}$ induced by $d$. 

We show first that
\begin{equation}
\label{eq:Haus-N-Dens}
\Hausd^n(\chain_k(\Gamma)) \geq N_{2\eps}(\chain_k(\Gamma))\cdot \Dens_\eps(\chain_k(\Gamma)).
\end{equation}
Let $N\in\bN$ and suppose that a set $\{y_1, y_2, \ldots, y_N\}\subset \chain_k(\Gamma)$ satisfies $\inf_{i\ne \ell}d_{k,\infty}(y_i,y_\ell) \geq 2\eps$. Since the sets $D_{k,\infty}(y_i,\eps)$, for $i=1,\ldots, N$, are mutually disjoint, we have 
\begin{align*}
\Hausd^n(\chain_k(\Gamma)) &\ge \Hausd^n\bigl((\chain_k(\Gamma)) \cap \bigcup_{i=1}^N D_{k,\infty}(y_i, \eps)\bigr) \\
&=\sum_{i=1}^N\Hausd^n\bigl((\chain_k(\Gamma)) \cap D_{k,\infty}(y_i, \eps)\bigr)\\
&\geq \sum_{i=1}^N \Dens_\eps(\chain_k(\Gamma)) =  N\cdot \Dens_\eps(\chain_k(\Gamma)).
\end{align*}
Thus \eqref{eq:Haus-N-Dens} follows.

Having \eqref{eq:Haus-N-Dens} at our disposal, we observe that, for each $\eps>0$, 
\begin{align*}
\lov(\Gamma)
& = \limsup_{k \to \infty} \frac{1}{k}\log(\Hausd^n(\chain_k(\Gamma)))\\
&\geq \limsup_{k \to \infty} \frac{1}{k}\big(\log(N_{2\eps}(\chain_k(\Gamma)))
+ \log(\Dens_\eps(\chain_k(\Gamma)))\big)\\
&\geq \limsup_{k \to \infty} \frac{1}{k} \log(N_{2\eps}(\chain_k(\Gamma)))
+ \liminf_{k \to \infty} \frac{1}{k} \log(\Dens_\eps(\chain_k(\Gamma)))\\
&= h_{2\eps}(\Gamma) + \lodn_\eps(\Gamma).
\end{align*}
Thus, \eqref{eq:h-lov-lodn} holds.
\end{proof}

\begin{rem}
The use of the product Riemannian distance in the definition of the Hausdorff $n$-measure stems from Theorem \ref{thm:Gromov_interpretation}. The above considerations hold also for the Hausdorff measures based on the metrics $d_{k,\infty}$.
\end{rem}

\subsection{Kolmogorov--Sinai entropy}\label{sec:KS}

In this section, we recall the necessary prerequisites of measure-theoretic entropy. We focus on the approach to the subject using measurable partitions. For a more in-depth discussion of this approach, see e.g.\ Przytycki--Urba\'nski \cite[Chapter 2]{Przytycki-Urbanski_book} or Rokhlin \cite{Rokhlin_entropy}.

Let $(X,\Sigma,\mu)$ be a complete probability Lebesgue space; for a precise definition, see e.g.\ \cite[Section 2.6]{Przytycki-Urbanski_book}. Note that, for a complete separable metric space $X$ and a Borel $\sigma$-algebra $\cB_X$ in $X$, the completion $(X, \cB_X^*, \mu^*)$ of a probability space $(X,\cB_X, \mu)$ is a Lebesgue space; see e.g.\ \cite[\S2, No.\ 7]{Rokhlin_measuretheory}. As usual, we denote $(X,\Sigma,\mu)$ by $X$ for simplicity.

Let $P_X$ be the set of all partitions of $X$. For each $\xi\in P_X$, and $x \in X$, we denote by $\xi(x)$ the unique element of $\xi$ containing $x$. We say that a partition $\eta\in P_X$ \emph{refines} the partition $\xi\in P_X$ if $\eta(x) \subset \xi(x)$ for every $x \in X$. The refinement of partitions induces a partial order $\le$ to the set $P_X$ of all partitions by $\xi\le\eta$ if $\eta$ refines $\xi$.

Given a partition $\xi\in P_X$, we say that a subset $A\subset X$ is a \emph{$\xi$-subset} if $A$ is a finite union of elements of $\xi\in P_X$. A partition $\xi \in P_X$ is \emph{measurable} if there exists an at most countable collection $(B_\alpha)_{\alpha\in I}$ of measurable $\xi$-subsets in $X$ having the following property:
\begin{quote}
For any distinct $C,C'\in\xi$, there exists $\alpha \in I$ for which either
\begin{itemize}
\item $C\subset B_\alpha$ and $C'\cap B_\alpha = \emptyset$, or 
\item $C'\subset B_\alpha$ and $C\cap B_\alpha = \emptyset$.
\end{itemize}
\end{quote}

\emph{Rokhlin's disintegration theorem} states that, if $X$ is a Lebesgue probability space and $\xi \in P_X$ is measurable, there exists a collection $((C,\Sigma|_C,\mu_C))_{C\in\xi}$ of probability spaces satisfying the following conditions:
\begin{itemize}
	\item[(a)] $(\xi(x),\Sigma|_{\xi(x)},\mu_{\xi(x)})$ is a Lebesgue space for $\mu$-a.e.\ $x \in X$, and
	\item[(b)] for any non-negative $\Sigma$-measurable function $f \colon X \to [0, \infty]$, the restriction $f\vert_{\xi(x)}$ is $(\Sigma \vert_{\xi(x)})$-measurable for $\mu$-a.e.\ $x \in X$, the function $x \mapsto \int_{\xi(x)} f\vert_{\xi(x)} \dd \mu_{\xi(x)}$ is $\Sigma$-measurable, and
	\begin{gather*}
		\int_X f \dd \mu 
		= \int_{X} \biggl(\int_{\xi(x)} f\vert_{\xi(x)} \dd \mu_{\xi(x)}
			\biggr) \dd \mu(x).
	\end{gather*}
\end{itemize}
For details, see e.g.\ \cite[Theorem 6.2.7, Remark 6.2.10]{Przytycki-Urbanski_book} and the surrounding discussion, or \cite[\S3]{Rokhlin_measuretheory}. The collection $((C,\Sigma|_C,\mu_C))_{C\in\xi}$, or in short $(\mu_C)_{C\in\xi}$, is called a \emph{canonical system of probability measures} associated to the space $X$ and partition $\xi$. The system $(\mu_C)_{C\in\xi}$ is essentially unique, in the sense that if $(\nu_C)_{C\in\xi}$ is another canonical system of probability measures associates to $X$ and $\xi$, then $\mu_{\xi(x)} = \nu_{\xi(x)}$ for $\mu$-a.e.\ $x \in X$.
 
Let $\xi, \eta \in P_X$ be measurable partitions. The \emph{conditional information function} $I_\mu(\xi \vert \eta) \colon X \to [0, \infty]$ of $\xi$ with respect to $\eta$ is defined by
\[
	I_\mu(\xi \vert \eta)(x) = - \log (\mu_{\eta(x)} (\xi(x) \cap \eta(x)))
\]
for $\mu$-a.e.\ $x \in X$, where $(\mu_C)_{C\in\eta}$ is a canonical system of probability measures associated to $X$ and $\eta$. The function $I_\mu(\xi \vert \eta)$ is $\Sigma$-measurable, and defines the \emph{conditional entropy $H_\mu(\xi\vert \eta)$ of $\xi$ with respect to $\eta$} by 
\begin{gather}\label{eq:conditionalentropy}
	H_\mu(\xi \vert \eta) := \int_{X} I_\mu(\xi \vert \eta) \dd\mu.
\end{gather}
For details, see e.g.\ \cite[Definition 2.8.3.\ and (2.8.3)]{Przytycki-Urbanski_book}.

For a sequence of measurable partitions $\xi_i \in P_X$, let $\bigvee_{j = 1}^{\infty} \xi_j$ denote the least common refinement of the partitions $\xi_i$, that is, the least partition $\zeta \in P_X$ satisfying $\xi_j \leq \zeta$ for every $j \in \Z_+$. This partition exists, and is measurable; see e.g.\ the discussion in \cite[pp.\ 39--40]{Przytycki-Urbanski_book}. Now, the \emph{measure-theoretic} or \emph{Kolmogorov--Sinai entropy} $h_\mu(f)$ of a measure-preserving self-map $f$ on a complete probability Lebesgue space $(X,\Sigma,\mu)$ is defined by
\begin{align}
 h_\mu(f)
:=&\sup_{\xi\in P_X:\text{ measurable}} H_\mu\biggl(\xi\biggm|\bigvee_{j = 1}^{\infty} f^{-j} \xi\biggr),
\end{align}
where $f^{-j} \xi = \left\{ f^{-j} C : C \in \xi\right\}$. The Kolmogorov--Sinai entropy is already determined by finite partitions, that is, 
\begin{align}
h_\mu(f) =&\sup_{\xi\in P_X:\text{ finite and measurable}} H_\mu\biggl(\xi\biggm|\bigvee_{j = 1}^{\infty} f^{-j} \xi\biggr).\label{metentropyrephrase}
\end{align}
Recall that a partition $\xi\in P_X$ is \emph{finite} if it has finitely many elements. For more details, see e.g.\ \cite[\S7 and \S9]{Rokhlin_entropy}.

Finally, we briefly comment on entropy in the case that $(X, \Sigma, \mu)$ is not a complete Lebesgue space. In this case, we still obtain a canonical system $(\mu_C)_{C \in \xi}$ of probability measures if $\xi$ is a finite partition of $X$ into $\Sigma$-measurable sets. Hence, the Kolmogorov--Sinai entropy of a measure-preserving $f \colon X \to X$ can be defined by
\begin{align}\label{metentropyfinite}
	h_\mu(f) =&\sup_{\xi\in P_X:\text{ finite and measurable}} \lim_{k \to \infty } H_\mu\biggl(\xi\biggm|\bigvee_{j = 1}^{k} f^{-j} \xi\biggr),
\end{align}
where the $\bigvee_{j = 1}^{k}$-operator is defined similarly as its infinite counterpart. Indeed, the limit in \eqref{metentropyfinite} always exists, and the result is equivalent with \eqref{metentropyrephrase} for complete Lebesgue spaces $X$; see e.g.\ \cite[Section 2.4 and Theorem 2.8.6]{Przytycki-Urbanski_book}. Moreover, we note that if $(X, \Sigma, \mu)$ is a probability space with completion $(X, \Sigma^*, \mu^*)$ and $f \colon X \to X$ is a $\mu$-preserving transformation, then $h_{\mu}(f) = h_{\mu^*}(f)$.

\section{Proof of the lower bound $h_{\mu_f}(f)\ge\log\deg f$}
\label{sec:lower_bound}

In this section, we prove the entropy lower bound. We formulate this goal as a proposition.

\begin{prop}
\label{prop:entropy_lower_bound}
Let $f\colon M\to M$ be a uniformly quasiregular map of degree at least $2$ on a closed, oriented, and connected Riemannian $n$-manifold $M$ satisfying $H^*(M) \not\cong H^*(\bS^n)$. Then
\[
h_{\mu_f}(f) \ge \log \deg f.
\]
\end{prop}

Recall that, by the variational principle, we have that
\[
h(f) \ge \sup_{\mu} h_\mu(f)
\]
for the topological entropy $h(f)$ of $f$, where the supremum is over $f$-invariant Borel probability measures $\mu$. Thus, Proposition \ref{prop:entropy_lower_bound} yields the desired lower bound in Theorem \ref{thm:main_result}.

Moreover, recall that a function $J_{f, \mu} \colon M \to \R$ is a \emph{(strong) measure theoretic Jacobian of $f$} with respect to a (Borel or completed Borel) measure $\mu$ on $M$ if, for every $\mu$-measurable set $A \subset M$ for which $f \vert_A$ is injective, the set $f(A)$ is $\mu$-measurable and the integral transformation formula
\[
	\int_A J_{f, \mu} \dd \mu = \mu(f(A))
\]
holds. For further information on measure theoretic Jacobians, see \cite[Section 2.9]{Przytycki-Urbanski_book}.

We prove the entropy estimate $h_{\mu_f}(f) \geq \log \deg f$ using the following lemma.

\begin{lemma} \label{lma:uniform} 
Let $f\colon M\to M$ be a uniformly quasiregular map of degree at least $2$ on a closed, oriented, and connected Riemannian $n$-manifold $M$. Let $\mu$ be an $f$-balanced Borel probability measure on $M$. Then the function
\[
	J_{f, \mu}(x) = \frac{\deg f}{i(x,f)}
\] 
is a measure-theoretic Jacobian of $f$ with respect to $\mu$.
\end{lemma}
\begin{proof}
	Suppose that $A \subset M$ is Borel and that $f\vert_A$ is injective. We decompose $A$ into sets $A_1, \dots, A_{\deg f}$, where $i(x, f) = j$ for every $x \in A_j$. These sets are again Borel due to the upper semicontinuity of $i(\cdot, f)$; see eg.\ \cite[Proposition I.4.10]{Rickman}. Similarly, the sets $f(A_j)$ are Borel since $f$ is continuous and finite-to-one; see e.g.\ \cite[Theorem 4.12.4]{Srivastava_Borel}. Finally, since $f \vert_A$ is injective, the sets $f(A_j)$ are disjoint.
	
	Now, let $j \in \{1, \dots, \deg f\}$ and let $x \in M$. Note that, since $f$ is injective on $A_j$ and $i(\cdot, f) \equiv j$ on $A_j$, we have $f_* \mathcal{X}_{A_j} = j \mathcal{X}_{f(A_j)}$, where $\mathcal{X}_{A_j}$ again denotes the characteristic function of a set $E \subset M$. By using the $f$-balanced property of $\mu$ and Lemma \ref{lma:pushfwd_set_measure}, it follows that
	\[
	(\deg f)\mu(A_j)
	= \int_M (\deg f) \mathcal{X}_{A_j} \dd \mu
	= \int_M f_* \mathcal{X}_{A_j} \dd \mu
	= j \mu(f(A_j)).
	\]
	Finally, we conclude that
	\[
	\int_A \frac{\deg f}{i(x,f)} \dd \mu(x)
	= \sum_{j=1}^{\deg f} \frac{(\deg f)\mu(A_j)}{j}
	= \sum_{j=1}^{\deg f} \mu(f(A_j))
	= \mu(f(A)),
	\]
	and the claim therefore follows.
\end{proof}

Having Lemma \ref{lma:uniform} at our disposal, we may conclude the proof of Proposition \ref{prop:entropy_lower_bound} as follows.

\begin{proof}[Proof of Proposition \ref{prop:entropy_lower_bound}]
	For simplicity, we implicitly complete the measure $\mu_f$ throughout this proof, as the completed measure has the same entropy as the original. Let	
	\[
		\eps_M := \{\{x\} : x \in M\}\in P_M,
	\]
	be the partition of $M$ into points. The partition $\eps_M$ and the partitions $f^{-j} \eps_M$ are $\mu_f$-measurable for $j \in \N$. Moreover, we note that $f^{-j}\eps_M \le f^{-1}\eps_M$ for every $j \in \N$, and therefore 
	\begin{equation}\label{preimage}
		\bigvee_{j= 1}^\infty f^{-j} \eps_M = f^{-1}\eps_M.
	\end{equation}
	
	By Lemma \ref{lma:uniform}, we obtain a measure theoretic Jacobian of $f$ with respect to $\mu_f$, given by $J_{f, \mu_f}(x) = (\deg f)/i(x, f)$ for $x \in M$. We note that since $\mu_f$ is $f$-balanced, $f$ maps $\mu_f$-null Borel sets to $\mu_f$-null Borel sets by Lemma \ref{lma:pushfwd_set_measure}. Therefore, $J_{f, \mu_f}$ remains a measure-theoretic Jacobian for the completed measure. 
	
	By \cite[Theorem 2.9.6]{Przytycki-Urbanski_book}, we hence obtain
	\begin{align*}
		H_{\mu_f} (\eps_M \mid f^{-1} \eps_M)
		& = \int_M \log J_{f, \mu_f} \dd \mu_f
		= \int_M \left(\log (\deg f) - \log i(x, f)\right) \dd \mu_f(x).
	\end{align*}
	Moreover, we have $i(x, f) = 1$ for every $x \notin B_f$. As previously discussed, due to \cite[Theorem 1.2]{Kangasniemi}, our assumption that $M$ is not a rational cohomology sphere implies that $\mu_f(B_f) = 0$; see Lemma \ref{lma:nobranchset}. Hence, we obtain that
	\begin{align*}
	\int_M \left(\log (\deg f) - \log i(x, f)\right) \dd \mu_f(x)
	&= \int_M \log \deg f \dd \mu_f
	= \log \deg f.
	\end{align*}
	Thus, by \eqref{metentropyrephrase} and \eqref{preimage}, we have
	\begin{gather*}
	h_{\mu_f}(f) \ge H_{\mu_f}\biggl( \eps_M \biggm \vert \bigvee_{j=1}^\infty f^{-1}\eps_M\biggr) =  H_{\mu_f}(\eps \mid f^{-1} \eps)=\log \deg f.
	\end{gather*}
\end{proof}

Moreover, we note that the only properties of $\mu_f$ we used in the above proof are that $\mu_f$ is an $f$-balanced Borel probability measure and $\mu_f(B_f) = 0$, and the latter assumption was only used to conclude that $\log i(\cdot, f)$ vanishes $\mu_f$-a.e. Hence, given the result of Lemma \ref{lma:uniform}, the above proof also yields the lower bound part of Theorem \ref{thm:main_result_v2}.

\section{Preliminaries on currents}\label{sect:Federer_Fleming}

We move now to the discussion of Gromov's argument on the upper bound $h(f) \le \log \deg f$ of the topological entropy. As a technical tool in the proof, we use Federer--Fleming currents and we recall some basic results in this section. We refer to Federer \cite[Chapter 4]{Federer} for details.

\subsection{Currents}
Let $U \subset \R^n$ be open, and for each $m\in\{0,1,\ldots,n\}$,
let $C^\infty_0(\wedge^m U)$ be the space of all differential $m$-forms on $U$
having coefficients in $C^\infty_0(U)$.
An \emph{$m$-current} on $U$ is an $\bR$-linear functional 
$T$ on $C^\infty_0(\wedge^m U)$ which is continuous in the sense of distributions. The space of all $m$-currents on $U$ is denoted by $\curr_m(U)$. We give $\curr_m(U)$ the topology of pointwise convergence.

The \emph{support} of a current $T \in \curr_m(U)$ is
\[
\spt T := U \setminus \cup \left\{ V \subset U: \text{ open, and } 
T(\omega) = 0\text{ for any }\omega \in C^\infty_0(\wedge^m V) \right\},
\]
and the \emph{boundary} $\partial_U T \in \curr_{m-1}(U)$ 
of an $m$-current $T \in \curr_m(U)$ is the $(m-1)$-current defined by
\[
\partial_U T(\omega) = T(d\omega)\quad \text{for each } \omega \in C^\infty_0(\wedge^{m-1} U).
\]
Thus $\partial_U \partial_U T = 0$ for any $T \in \curr_m(U)$. 
For each $c \in \R$, the multiplication $c T$ is defined in the obvious manner. Furthermore, for each $l$-form $\tau \in C^\infty(\wedge^l U)$ for $l\in\{0,\ldots,m\}$, the \emph{interior multiplication} $T \llcorner \tau \in \curr_{m-l}(U)$ is the current defined by $(T \llcorner \tau)(\omega) = T(\tau \wedge \omega)$ for each $\omega \in C^\infty_0(\wedge^{m-l} U)$.

\subsection{The mass of currents, normal currents, and integral representations}

Let $W$ be an $n$-dimensional $\bR$-vector space having
an inner product $\ip{\cdot}{\cdot}$. For each $m\in\{1,\ldots,n\}$,
the $m$-th exterior product space (the $m$-vector space) 
$\wedge^m W$ of $W$ is equipped with the Grassmann inner product
 \[
 \ip{v_1 \wedge \cdots \wedge v_m}{w_1 \wedge \cdots \wedge w_m} 
 = \det \left(\ip{v_i}{w_j}\right)_{i,j}\quad \text{for } v_i, w_j \in W. 
 \]
We denote the induced norm on $\wedge^m W$ by $|\cdot|$. 
The $m$-covector space $\wedge^m W^*$ of $W$ also has a Grassmannian inner product and a norm induced by the duality isomorphism $W \to W^*$ given by $v \mapsto (w \mapsto \ip{v}{w})$ for $v, w \in W$. 
 
 The \emph{comass} $\|\xi\|_{\mass}$ 
 of an $m$-covector $\xi \in \wedge^m W^*$ is defined by
 \[
 \norm{\xi}_{\mass} 
 := \sup \{|\xi(w)|: w \in \wedge^m W \text{ is simple}, |w| \leq 1 \},
 \]
 where we say an $m$-vector $w \in \wedge^m W$ is {\em simple} 
if it can be written as $w = w_1\wedge \cdots \wedge w_m$. Similarly, the \emph{mass} $\norm{w}_{\mass}$ of 
 an $m$-vector $w \in \wedge^m W$ is defined by
 \[
 \norm{w}_{\mass}:= \sup\{|\xi(w)|: \xi \in \wedge^m W^*, \norm{\xi}_{\mass}\leq 1 \}.
 \]
These are norms on $\wedge^m W^*$ and $\wedge^m W$ satisfying
$|\xi| \geq \norm{\xi}_{\mass}$ for any $\xi \in \wedge^m W^*$  
and $|w| \leq \norm{w}_{\mass}$
for any $w \in \wedge^m W$, respectively. 
For more details, see \cite[Section 1.8]{Federer}.

Let $U$ be an open set in $\bR^n$ and $m\in\{0,1,\ldots,n\}$.
For each open subset $V\subset U$,
the \emph{mass of an $m$-current $T \in \curr_m(U)$ over $V$} is defined by
\[
\mass_V(T) := \sup \Bigl\{ |T(\omega)| :
		\omega \in C^\infty_0(\wedge^m V), 
		\sup_{x \in V} \norm{\omega_x}_{\mass} \leq 1 \Bigr\},	
\]
where 
$C^\infty_0(\wedge^m V)$ is embedded in $C^\infty_0(\wedge^m U)$ by means of 
zero extension on $U\setminus V$. 
An $m$-current $T \in \curr_m(U)$ is said to be \emph{normal} if 
\[
\spt T\Subset U\quad\text{and}\quad
\mass_U(T) + \mass_U(\partial_U T) < \infty;
\]
here, and in what follows, we denote $A\Subset B$ if $A$ is a subset compactly contained in $B$.

An $m$-current $T \in \curr_m(U)$ is \emph{locally normal} if 
$\mass_V(T) + \mass_V(\partial_U T) < \infty$
for any open subset $V\Subset U$.
Let $\ncurr_m(U)$ (resp.\ $\ncurr_m^\loc(U)$) be
the space of all normal (resp.\ locally normal) $m$-currents on $U$. 

Currents of finite mass admit an integral representation.

\begin{lemma}\label{lem:mass_integral_representation}
For every $T \in \curr_m(U)$ satisfying $\mass_U(T)<\infty$, 
there exist a measurable tangent $m$-vector field $\vec{T}$ on $U$ 
and a Radon measure $\mu_T$ on $U$ such that
for every $\omega \in C_0^\infty(\wedge^m U)$,
\begin{gather}
  T(\omega) = \int_{U} \langle\omega,\vec{T}\rangle\dd \mu_T.\label{eq:integral}
\end{gather}
Moreover, $\mu_T(V) = \mass_{V}(T)$ for any open subset $V \subset U$.
\end{lemma}

Let $T\in \curr_m(U)$ be a current of finite mass on an open set $U\subset \R^n$ and let $\mu_T$ be a Radon measure and $\vec{T}$ an $m$-vector field representing $T$ as in \eqref{eq:integral}.
Thus, for an open set $V\subset U$, we may define the $m$-current
$T \llcorner V=T\llcorner\chi_V \in \curr_m(U)$ by
\begin{gather}\label{eq:restriction} 
 (T \llcorner V) (\omega) = \int_{V} \langle\omega,\vec{T}\rangle\dd \mu_T
=\int_U\chi_V\cdot\langle\omega,\vec{T}\rangle\dd\mu_T
\end{gather}
for each $\omega \in C_0^\infty(\wedge^m U)$, where $\chi_V$ is the characteristic function of $V$ on $U$. 
Moreover, 
\begin{equation}\label{eq:mass_restriction_interplay}
	\mass_{V}(T) = \mu_T(V) = (\mu_T|V)(V) = (\mu_T|V)(U)
	= \mass_{U} (T \llcorner V).
\end{equation}
For further details, we refer to \cite[Sections 4.1.5 and 4.1.7]{Federer}

\subsection{Push-forward of currents}\label{subsect:pushforward_of_currents}

Let $U \subset \R^{n_1}$ and $V \subset \R^{n_2}$ be open, $T \in \curr_m(U)$, and 
let $h \colon U \to V$ be a smooth map such that the restriction 
$h|\spt T:\spt T\to V$ is proper; note that, if $\spt T\Subset U$, then $h \vert_{\spt T}$ is proper. The \emph{push-forward $h_*T$ of $T$ under the map $h$} is the $m$-current $h_*T\in \curr_m(V)$ defined as follows. 
For every $\omega \in C_0^\infty(\wedge^m V)$, let $\psi \in C_0^\infty(U)$ be a function satisfying $\psi\equiv 1$ on some open neighborhood of $(\spt T)\cap(h^{-1} \spt \omega)$, and set
\[
(h_*T)(\omega) = T(\psi\cdot h^* \omega).
\] 
The values of $h_*T$ are independent on the choice of $\psi$.

Since $h^*d\omega = dh^* \omega$ for any $\omega \in C^\infty(\wedge^m V)$, we have
\[
h_*\partial_U T = \partial_V h_*T\quad \text{for each }T \in \curr_m(U). 
\]
If in addition $h \vert_{\spt T}$ is $L$-Lipschitz for $L\ge 1$, then
for any $T \in \curr_m(U)$,
\[
\mass_V(h_*T) \leq L^m \mass_U(T).
\]
For more details, we refer to e.g.\ \cite[sections 4.1.7 and 4.1.14]{Federer}.

\subsection{Slicing of currents}\label{subsect:slicing}

Let $U \subset \R^n$ be an open set, $T \in \curr_m(U)$ an $m$-current satisfying $\mass_U(T)+\mass_U(\partial_U T) < \infty$, and let $h \colon U \to \R$ be an $L$-Lipschitz function for $L\ge 1$. 
For each $t\in\bR$, we set
\begin{gather*}
 U_{h,t} := h^{-1} (-\infty, t)\subset U, 
\end{gather*}
which is open, and the \emph{slice of $T$ by $h$ at $t$} is
\[
\langle T, h, t- \rangle := \partial_U (T \llcorner U_{h,t}) 
		- (\partial_U T) \llcorner U_{h,t}\in \curr_{m-1}(U).
\]

The following lemma gathers the key properties of the slices of currents
used in the forthcoming discussion.
The argument in the proof is
similar to that in \cite[Section 4.2.1]{Federer} and we omit the details.

\begin{prop}\label{lem:properties_of_slices}
Let $U \subset \R^n$ be open, let $h$ be an $L$-Lipschitz 
function on $U$, $L\ge 1$, and let $(a,b)\subset\bR$. If 
$\emptyset\neq U_{h,t} \Subset U$ for every $t \in (a,b)$, then 
for every $T \in \curr_m(U)$ satisfying 
$\mass_U(T)+\mass_U(\partial_U T) < \infty$, 
 \begin{enumerate}
  \item \label{enum:slice_prop_2} $\langle T, h, t- \rangle \in \ncurr_{m-1}(U)$ for 
  Lebesgue a.e.\ $t \in (a,b)$, and
  \item \label{enum:slice_prop_1} 
the function $t \mapsto \mass_{U} (\langle T, h, t- \rangle)$ on $(a,b)$
is lower semicontinuous, and
	\[
	\mass_{U_{h,t}}(T)  
	\geq \frac{1}{mL} \int_a^t \mass_{U} (\langle T, h, s- \rangle) \dd s
	\quad\text{for } t \in (a,b).
	\]
 \end{enumerate}
\end{prop}

\section{The Ahlfors regularity of images in Euclidean spaces}\label{sec:Ahlfors_euclidean}

As mentioned in the introduction, the upper bound for the entropy $h(f)$ follows from an application of the uniform Ahlfors regularity estimate in Theorem \ref{thm:Gromov_interpretation} to the images of maps $(\id, f, \ldots, f^k) \colon M\to M^{k+1}$. We begin by proving a Euclidean counterpart of Theorem \ref{thm:Gromov_interpretation}. For the statement, given $\Gamma \subset (\R^n)^k$, we denote
\[
\Gamma_{y,r} = B^{kn}(y,r) \cap \Gamma
\]
for $y\in \R^{nk}$ and $r>0$.

\begin{prop}\label{prop:Gromov_interpretation_Euclidean}
Let $\Omega\subset\bR^n$ be an open subset for $n \geq 2$,
$k\in\bN$, and let $f_1, \ldots, f_k \colon \Omega \to \R^n$ 
be non-constant $K$-quasiregular maps for some $K\ge 1$
such that $\max_{j\in\{1,\ldots,k\}}N(f_j) < \infty$. 
Let $g:= (f_1, \ldots, f_k) \colon \Omega \to (\bR^n)^k=\R^{kn}$ and 
$\Gamma=\Gamma_g:=g(\Omega)\subset\R^{kn}$.
Then there exists a constant $C=C(n)>0$, depending only on $n$, having the property that, for each $y \in \Gamma$ and any $r>0$ satisfying
$g^{-1}(\Gamma_{y,r})\Subset\Omega$, we have
\begin{equation}
\label{eq:Ahlfors}
 \frac{1}{Ck^\frac{n(n-1)}{2} K^{n-1} \left(\min_j N(f_j)\right)^{n}}
\le\frac{\cH^n(\Gamma_{y,r})}{r^n}
\leq Ck^\frac{n}{2} K\max_j N(f_j).
\end{equation}
\end{prop}

We prove Proposition \ref{prop:Gromov_interpretation_Euclidean} following Gromov's argument in \cite{Gromov-2003}.
For the rest of this section, let $g\colon \Omega \to \R^{kn}$ be a map as in 
Proposition \ref{prop:Gromov_interpretation_Euclidean}.
The map $g \colon \Omega \to \R^{nk}$ is continuous and 
in $W^{1,n}_\loc(\Omega,\bR^{nk})$. 
As previously, we set
\begin{gather*}
 N(g, y, A):=\#(g^{-1}\{y\}\cap A)
\end{gather*}
for each $y\in\Gamma$ and each $A\subset\Omega$, 
$N(g,y):=N(g,y,\Omega)$ for each $y\in\Gamma$, and 
\[
N(g):=\sup_{y\in\Gamma}N(g,y)\le\min_{j\in\{1,\ldots,k\}}N(f_j)<\infty.
\]

For each $j\in\{1,\ldots,k\}$, 
let $\proj_j\colon (\R^n)^k \to \R^n$ be the $j$-th projection $(z_1,\ldots, z_k) \mapsto z_j$. Then ${\proj_j}\circ g = f_j$.

We define a measurable function $|J_g|$ on $\Omega$ by
\begin{gather*}
 |J_g|(x)=|(Dg(x) e^1_{x}) \wedge \cdots \wedge (Dg(x)e^n_{x})|\quad
\text{for Lebesgue a.e.\ }x\in\Omega,
\end{gather*}
where $(e^1_x,\ldots, e^n_x)$ is the standard basis of $T_x\Omega$. Note that, for $k>1$, the map $g \colon \Omega \to \R^{nk}$ does not have a well-defined Jacobian determinant $J_g$. We call the function $|J_g|$ the \emph{$n$-Jacobian of $g$}.

\subsection{The upper Ahlfors bound}

The upper bound for $\Hausd^n(\Gamma_{y,r})$ follows from the measures $\Hausd^n(\proj_j(\Gamma_{y,r}))$ of the projections $\proj_j(\Gamma_{y,r})$ and the multiplicity of the restrictions $\proj_j|\Gamma_{y,r}$, which in turn can be estimated in terms of the multiplicity of the maps $f_j$. We formulate this as a lemma.

\begin{lemma}\label{lem:upper_hausdorff_bound}
Let $\Omega\subset \R^n$ be an open subset, $f_1,\ldots, f_k \colon \Omega \to \R^n$ be $K$-quasiregular mappings, $g = (f_1,\ldots, f_k) \colon \Omega \to \R^{nk}$, and $\Gamma = g(\Omega) \subset \R^{nk}$.
Then for every open subset $U \subset \Gamma$ satisfying $g^{-1}U\Subset\Omega$, we have
\begin{gather}\label{eq:projection_estimate}
  \Hausd^n(U) \leq n^{\frac{n}{2}}k^{\frac{n}{2}}K
\max_{j\in\{1,\ldots,k\}}N(f_j) \Hausd^n(\proj_j(U)).
\end{gather}
\end{lemma}

The upper bound in Proposition \ref{prop:Gromov_interpretation_Euclidean} follows now immediately. Indeed, since $\Gamma_{y,r}$ is open in $\Gamma$, we have, by \eqref{eq:projection_estimate}, that
\begin{align*}
\Hausd^n(\Gamma_{y,r}) &\le n^{\frac{n}{2}}k^{\frac{n}{2}}K
\max_{j\in\{1,\ldots,k\}}N(f_j) \Hausd^n(\proj_j(\Gamma_{y,r})) \\
&\le n^{\frac{n}{2}}k^{\frac{n}{2}}K 
\max_{j\in\{1,\ldots,k\}}N(f_j)\Hausd^n(B^n(\proj_j(y),r)) \\
&\le C(n) k^{\frac{n}{2}}K \left( \max_{j\in\{1,\ldots,k\}}N(f_j) \right) r^n, 
\end{align*}
where $C(n)>0$ depends only on $n$. Thus it suffices to prove Lemma \ref{lem:upper_hausdorff_bound}.

We begin by showing that the map $g$ has the Lusin property. 
\begin{lemma}\label{lem:Lusin_n}
Let $\Omega\subset \R^n$ be an open subset, $f_1,\ldots, f_k \colon \Omega \to \R^n$ be $K$-quasiregular mappings, $g = (f_1,\ldots, f_k) \colon \Omega \to \R^{nk}$, and $\Gamma = g(\Omega) \subset \R^{nk}$.
If $E \subset \Omega$ is an $\Hausd^n$-null subset, then 
$g(E) \subset \R^{nk}$ is also an $\Hausd^n$-null subset.
\end{lemma}

\begin{proof}
For each $j\in\{1,\ldots,k\}$,
the $j$-th component $f_j$ of $g=(f_1,\ldots,f_k)$ is quasiregular, 
and we may therefore fix an exponent $p_j > n$ of local higher integrability for $Df_j$. 
Then the proof of Bojarski and Iwaniec in \cite[Section 8.1]{Bojarski-Iwaniec} shows that 
there is $C(n,p_j)>0$ depending only on $n,p_j$ such that
if $Q_i \subset \Omega$ are disjoint cubes, then
\[
	\sum_i\left(\diam f_j(Q_i)\right)^n
 	\leq C(n, p_j) \Hausd^n\Bigl(\bigcup_i Q_i\Bigr)^{1-\frac{n}{p_j}} 
 	\biggl(\int_{\bigcup_i Q_i} |Df_j|^{p_j} \biggr)^\frac{n}{p_j}.
\]
Pick a common exponent $p > n$ of higher integrability for all $Df_j$, $j\in\{1,\ldots,k\}$. Then by H\"older's inequality and standard estimates, there exists $C(n,k,p)>0$ depending only on $n,k,p$ such that if $Q_i \subset \Omega$ are cubes with disjoint interiors, then
\begin{align*}
  	\sum_i \left(\diam g(Q_i)\right)^n
  	\leq C(n, k, p) \Hausd^n\Bigl(\bigcup_i Q_i\Bigr)^{1-\frac{n}{p}} 
  	\biggl( \sum_{j=1}^k \int_{\bigcup_i Q_i}|Df_j|^p \biggr)^\frac{n}{p}.
\end{align*}
Now the proof of the Lusin condition follows by intersecting the set of zero measure $E$ with a compact subset $A \subset \Omega$, covering $E \cap A$ with a collection of cubes with disjoint interiors and arbitrarily small total measure, and using the above estimate
to show that $g(E \cap A)$ has arbitrarily small $\Hausd^n$ measure.
\end{proof}

Since the maps $f_j \colon \Omega \to\R^n$ are $K$-quasiregular, we have the following estimate for the $n$-Jacobian of $g = (f_1,\ldots, f_k) \colon \Omega \to \R^{nk}$.

\begin{lemma}\label{lem:area_pontwise_estimate}
Let $\Omega\subset \R^n$ be an open subset, $f_1,\ldots, f_k \colon \Omega \to \R^n$ be $K$-quasiregular mappings, $g = (f_1,\ldots, f_k) \colon \Omega \to \R^{nk}$, and $\Gamma = g(\Omega) \subset \R^{nk}$. Then, for Lebesgue almost every $x \in \Omega$, we have
\[
|J_g|(x)\leq n^{\frac{n}{2}}K k^{\frac{n}{2}-1}\sum_{j=1}^kJ_{f_j}(x).
\]
\end{lemma}

\begin{proof}
Since
\[
|J_g|(x) = \sqrt{\det((Dg(x))^T Dg(x))},
\]
we have, by the distortion bound \eqref{eq:bd} for $f_j$ 
and H\"older's inequality, that
\begin{align*}
	|J_g|(x) 
	&=\sqrt{\det\sum_{j=1}^{k}(Df_j(x))^T Df_j(x)}
	\le\sqrt{\biggl(\frac{1}{n}\tr\sum_{j=1}^{k}(Df_j(x))^T Df_j(x)\biggr)^n}\\
	&\leq \frac{1}{n^{\frac{n}{2}}} 
 	\biggl( \sum_{j=1}^{k} \tr\bigl((Df_j(x))^T Df_j(x)\bigr) \biggr)^\frac{n}{2}
 	\leq \frac{1}{n^{\frac{n}{2}}} 
 		\biggl( \sum_{j=1}^k(n\norm{Df_j(x)})^2\biggr)^\frac{n}{2}\\
 	&\leq n^{\frac{n}{2}}K\cdot
		\biggl(\sum_{j=1}^k(J_{f_j}(x)^{\frac{2}{n}})^{\frac{n}{2}}
			\biggr)^{\frac{2}{n}\cdot\frac{n}{2}}
		k^{(1-\frac{2}{n})\frac{n}{2}}
	= n^{\frac{n}{2}}Kk^{\frac{n}{2}-1} \sum_{j=1}^k J_{f_j}(x)
\end{align*}
for Lebesgue a.e.\ $x\in\Omega$. 
\end{proof}

The last ingredient is the proof of Lemma \ref{lem:upper_hausdorff_bound} is an area formula for $g$. For more details, see Haj\l asz 
{\cite[Theorem 11]{Hajlasz-areathm}}.

\begin{lemma}\label{lem:area_thm}
Let $\Omega\subset \R^n$ be an open subset, $f_1,\ldots, f_k \colon \Omega \to \R^n$ be $K$-quasiregular mappings, $g = (f_1,\ldots, f_k) \colon \Omega \to \R^{nk}$, and $\Gamma = g(\Omega) \subset \R^{nk}$. Then, for every open subset $A \subset \Omega$, 
 \begin{equation}\label{eq:area_formula}
  \int_{A}|J_g|\dd\Hausd^n   = \int_{g(A)} N(g, y, A) \dd\Hausd^n(y).
\end{equation}
\end{lemma}
\begin{proof}
The map $g$ is in $W^{1,n}_\loc(\Omega,\bR^{kn})$, 
and by Lemma \ref{lem:area_pontwise_estimate}, we have
$|J_g|\in L^1_{\operatorname{loc}}(\Omega)$. 
Hence, the Sobolev area formula \cite[Theorem 11]{Hajlasz-areathm} implies that \eqref{eq:area_formula} holds for some $\tilde{g}$ in the Sobolev equivalence class of $g$. Moreover, since $g$ is Lusin (N) by Lemma \ref{lem:Lusin_n}, we have $\tilde{g}= g$ by the discussion in \cite[p. 239]{Hajlasz-areathm}.
\end{proof}

We are now ready for the proof of Lemma \ref{lem:upper_hausdorff_bound}.

\begin{proof}[Proof of Lemma \ref{lem:upper_hausdorff_bound}]
Let $U \subset \Gamma$ be an open set satisfying $g^{-1}U \Subset \Omega$. Then, for each $y\in U$, we have $N(g,y,U) \ge 1$. Thus, by Lemmas \ref{lem:area_thm} and \ref{lem:area_pontwise_estimate}, we have
 \begin{align*}
  \Hausd^n(U) &\leq \int_U N(g, y,U) \dd \Hausd^n(y)
  = \int_{g^{-1}U} |J_g|\dd \Hausd^n\\
  &\leq n^{\frac{n}{2}}K k^{\frac{n}{2}-1}
  \sum_{j=1}^k \int_{g^{-1}U} J_{f_j}(x) \dd \Hausd^n(x),
\end{align*}
Since $f_j= \proj_j\circ g$, the change of variables for quasiregular mappings yields
\begin{align*}
\sum_{j=1}^k \int_{g^{-1}U} J_{f_j} \dd \Hausd^n
&= \sum_{i=1}^k \int_{\proj_j(U)} N(f_j, y_j, g^{-1} U) \dd \Hausd^n(y_j)\\
&\leq \sum_{j=1}^k N(f_j) \Hausd^n(\proj_j(U))\leq k\max_{j\in\{1,\ldots,k\}} N(f_j)\Hausd^n(\proj_j(U)),
\end{align*}
which completes the proof.
\end{proof}

\subsection{The lower Ahlfors bound}

In this section, we prove the lower estimate in Proposition \ref{prop:Gromov_interpretation_Euclidean}. The lower bound is obtained by considering a current $[\Gamma_{y,r}]$ associated to $\Gamma_{y,r}$ and two estimates which we combine in the following proposition. We define the current $[\Gamma_{y,r}]$ after the statement and devote the rest of this section for the proofs of the estimates.

\begin{prop}\label{prop:lower_Ahlfors_Euclidean}
Let $\Omega\subset\bR^n$ be an open subset for $n \geq 2$,
$k\in\bN$, and let $f_1, \ldots, f_k \colon \Omega \to \R^n$ 
be non-constant $K$-quasiregular maps for some $K\ge 1$
such that $\max_{j\in\{1,\ldots,k\}}N(f_j) < \infty$. 
Let $g = (f_1, \ldots, f_k) \colon \Omega \to \R^{kn}$ and 
$\Gamma=g(\Omega)\subset\R^{kn}$.
Then there exists a constant $C=C(n)>0$ depending only on $n$ having the property that
\begin{equation}
\label{eq:lower_Ahlfors_Euclidean}
\left( \min_{j\in \{1,\ldots, k\}} N(f_j)\right) \Hausd^n(\Gamma_{y,r}) \ge \mass_{\R^{kn}}\left([\Gamma_{y,r}]\right) \ge \left( \frac{1}{C k^{\frac{n}{2}} K \min_j N(f_j)} \right)^{n-1} r^n
\end{equation}
for each $y\in \Gamma$ and $r>0$ for which $g^{-1}(\Gamma_{y,r}) \Subset \Omega$.
\end{prop}

The lower bound in \eqref{eq:Ahlfors} follows immediately from this lemma and hence the proof of this proposition completes the proof of Proposition \ref{prop:Gromov_interpretation_Euclidean}.

\subsubsection{Current $[\Gamma_{y,r}]$}

Although the notation may suggest otherwise, we do not define the current $[\Gamma_{y,r}]$ as integration over $\Gamma_{y,r}$ but a push-forward of the integration over $g^{-1}(\Gamma_{y,r})$.

Let $y\in \Gamma$ and $r>0$ be such that $\Omega_{y,r}=g^{-1}(\Gamma_{y,r}) \Subset \Omega$. In this case, $g|\Omega_{y,r} \colon \Omega_{y,r} \to B^{nk}(y,r)$ is a proper map. Indeed, let $S \subset B^{nk}(y,r)$ be compact. Then $g^{-1}S$ is a closed subset of $\Omega$. Moreover $g^{-1}S \subset g^{-1}B^{nk}(y,r) = \Omega_{y,r} \subset \overline{\Omega}_{y,r}$. Since $\overline{\Omega}_{y,r} \subset \Omega$, we have that $g^{-1}S = g^{-1}S \cap \overline{\Omega}_{y,r}$ is a closed subset of $\overline{\Omega}_{y,r}$ by relative topology. Since $\overline{\Omega}_{y,r}$ is compact, $g^{-1}S$ is compact.

Since $g \in  W^{1,n}_\loc(\Omega,\bR^{nk})$, the linear functional $[\Gamma_{y,r}] \colon C^\infty_0(\wedge^n \R^{kn}) \to \R$,
\begin{equation}\label{eq:pushed_integration_current}
\omega \mapsto \int_{\Omega_{y,r}} g^*\omega,
\end{equation}
where $g^*\omega$ is a measurable $n$-form in $\Omega$, is well-defined.

To show that $[\Gamma_{y,r}]$ is a current, denote, for every $\omega \in C^\infty_0(\wedge^n\bR^{kn})$,
\begin{gather*}
\lambda_\omega:=\sup_{x\in\R^{kn}}\norm{\omega_x}_{\mass}<\infty.
\end{gather*}
Then, for Lebesgue a.e.\ $x\in\Omega$,
\begin{equation}
\begin{split}
|(g^* \omega)_{x}| &=|\bigl\langle(g^*\omega)_{x},\vol_{\R^n}(x)\bigr\rangle| \\
  &= |\omega_{g(x)}\wedge\left( (Dg(x) e^1_{x}) \wedge \cdots \wedge (Dg(x)e^n_{x}) \right)| \\
  &\leq \norm{\omega_{g(x)}}_{\mass}|J_g(x)|\le\lambda_\omega|J_g(x)|.
\end{split}
\label{eq:pointwise}
 \end{equation}

We are now ready to prove the upper bound in Proposition \ref{prop:lower_Ahlfors_Euclidean}.

\begin{lemma}\label{lem:well_defined_current}
	Let $\Omega\subset\bR^n$ be an open subset for $n \geq 2$,
	$k\in\bN$, and let $f_1, \ldots, f_k \colon \Omega \to \R^n$ 
	be non-constant $K$-quasiregular maps for some $K\ge 1$
	such that $\max_{j\in\{1,\ldots,k\}}N(f_j) < \infty$. 
	Let $g = (f_1, \ldots, f_k) \colon \Omega \to \R^{kn}$ and 
	$\Gamma=g(\Omega)\subset\R^{kn}$.
	Let $y\in \Gamma$ and $r>0$ be such that $g^{-1}(\Gamma_{y,r})\Subset \Omega$. Then the functional $[\Gamma_{y,r}]$ is a current in $\curr_n(\R^{kn})$ and 
 	\begin{gather} \label{eq:mass}
  	\mass_{\R^{kn}}\left([\Gamma_{y,r}]\right) 
   		\leq \Bigl( \min_{j\in\{1,\ldots,k\}}N(f_j) \Bigr)\cH^n(\Gamma_{y,r})<\infty.
 	\end{gather}
\end{lemma}

\begin{proof}
For every $\omega \in C^\infty_0(\wedge^n\bR^{kn})$, we have, by \eqref{eq:pointwise} and Lemma \ref{lem:area_thm}, that
\begin{align*}
\Bigl|\int_{\Omega_{y,r}} g^*\omega\Bigr|
&\leq \lambda_\omega\int_{\Omega_{y,r}}|J_g|\dd \Hausd^n
= \lambda_\omega \int_{\Gamma_{y,r}} N(g, y') \dd\Hausd^n(y')\\
&\leq \lambda_\omega  N(g) \Hausd^n(\Gamma_{y,r})
\leq \lambda_\omega \Bigl(\min_{j\in\{1,\ldots,k\}} N(f_j)\Bigr)\Hausd^n(\Gamma_{y,r}).
\end{align*}
To show that $[\Gamma_{y,r}]$ is a current it suffices now to observe that, for a converging sequence $\omega_j \to 0$ in $C^\infty_0(\wedge^n\bR^{kn})$, we have
\begin{align}\label{eq:pre-mass}
\left| [\Gamma_{y,r}](\omega_j)\right| 
\le \lambda_{\omega_j}  \Bigl(\min_{j\in\{1,\ldots,k\}} N(f_j)\Bigr)\Hausd^n(\Gamma_{y,r}).  
\end{align}
Since differential forms are sections of covectors, we have the point-wise estimate $\norm{(\omega_j)_x}_{\mass} \le \lvert(\omega_j)_x\rvert$ for almost every $x\in \Omega$. Thus $\lambda_{\omega_j} \to 0$ as $j\to \infty$. Since $\Hausd^n(\Gamma_{y,r}) < \infty$ by \eqref{eq:projection_estimate}, $[\Gamma_{y,r}]$ is continuous and hence a current. Moreover, the mass estimate \eqref{eq:mass} follows from the estimate \eqref{eq:pre-mass}. 
\end{proof}

We move now to prove the lower bound in Proposition \ref{prop:lower_Ahlfors_Euclidean}. We begin by proving that the current $[\Gamma_{y,r}]$ is locally normal.

\begin{lemma}\label{lem:local_normality}
Let $\Omega\subset\bR^n$ be an open subset for $n \geq 2$,
$k\in\bN$, and let $f_1, \ldots, f_k \colon \Omega \to \R^n$ 
be non-constant $K$-quasiregular maps for some $K\ge 1$
such that $\max_{j\in\{1,\ldots,k\}}N(f_j) < \infty$. 
Let $g = (f_1, \ldots, f_k) \colon \Omega \to \R^{kn}$ and 
$\Gamma=g(\Omega)\subset\R^{kn}$.
Let $y\in \Gamma$ and $r>0$ be such that $g^{-1}(\Gamma_{y,r})\Subset \Omega$. Then 
\[
\partial_{B^{kn}(y,r)} [\Gamma_{y,r}]=0.
\]
\end{lemma}
\begin{proof}
Since $\Omega_{y,r} = g^{-1}(\Gamma_{y,r}) \Subset \Omega$, the map $g|\Omega_{y,r}$ is also in $W^{1,n}(\Omega_{y,r},\bR^{kn})$. 
Hence, by e.g.~\cite[Proposition 4.1]{Goldstein-Troyanov_cohomology},
for every $\omega \in C_0^\infty(\wedge^{n-1} B^{kn}(y,r))$, we have
$\dd g^*\omega = g^*\dd\omega \in L^1(\wedge^n \Omega_{y,r})$
and $g^*\omega \in L^{n/(n-1)}(\wedge^{n-1} \Omega_{y,r})$,
where $\dd g^*\omega$ is defined in the weak sense, that is, 
\[
\int_{\Omega_{y,r}} \psi\dd g^*\omega 
= -\int_{\Omega_{y,r}}\dd\psi \wedge g^*\omega
\]
for every $\psi \in C_0^\infty(\Omega_{y,r})$.

Since $g^*\omega$ is compactly supported in $\Omega_{y,r}$, there exists, by a standard convolution argument, a sequence $(\omega_j)$ of $(n-1)$-forms in $C_0^\infty(\wedge^{n-1} \Omega_{y,r})$ for which $\omega_j \to g^*\omega$ in $L^{n/(n-1)}(\wedge^{n-1}\Omega_{y,r})$ and $\dd\omega_j \to \dd g^*\omega = g^*\dd\omega$ in $L^1(\wedge^n\Omega_{y,r})$ as $j \to \infty$.
Thus
\begin{gather*}
 \partial_{B^{kn}(y,r)} [\Gamma_{y,r}](\omega)=[\Gamma_{y,r}](\dd\omega)
=\int_{\Omega_{y,r}}g^*\dd\omega
=\lim_{j\to\infty}\int_{\Omega_{y,r}}\dd\omega_j=0,
\end{gather*}
that is, the boundary $\partial_{B^{kn}(y,r)} [\Gamma_{y,r}]$ vanishes.
\end{proof}

Currents $[\Gamma_{y,r}]$ restrict naturally to currents $[\Gamma_{y,t}]$ for $t\in (0,r)$.

\begin{lemma}\label{lem:restriction_formula}
Let $\Omega\subset\bR^n$ be an open subset for $n \geq 2$,
$k\in\bN$, and let $f_1, \ldots, f_k \colon \Omega \to \R^n$ 
be non-constant $K$-quasiregular maps for some $K\ge 1$
such that $\max_{j\in\{1,\ldots,k\}}N(f_j) < \infty$. 
Let $g = (f_1, \ldots, f_k) \colon \Omega \to \R^{kn}$ and 
$\Gamma=g(\Omega)\subset\R^{kn}$.
Let $y\in \Gamma$ and $r>0$ be such that $g^{-1}(\Gamma_{y,r}) \Subset \Omega$. Then, for every $t\in(0,r)$, 
\[
[\Gamma_{y,r}] \llcorner B^{kn}(y, t) = [\Gamma_{y,t}].
\]
\end{lemma} 
\begin{proof}
Let $t\in (0,r)$ and let $(A_i)$ be an increasing sequence of compact subsets exhausting $B^{kn}(y, t)$, that is, $B^{kn}(y, t) = \bigcup_{j=1}^\infty A_i$. 
For every $i\in\bN$, let $\psi_i \in C_0^\infty(B^{kn}(y, t))$ be a function for which $0\le\psi_i \leq 1$ and $\psi_i|_{A_i}= 1$. 
Then
\begin{gather*}
 ([\Gamma_{y,r}] \llcorner B^{kn}(y, t))\llcorner \psi_i
 = [\Gamma_{y,r}]\llcorner\psi_i.
\end{gather*}

Let 
\[
[\Gamma_{y,r}](\cdot) = \int_{\Omega_{y,r}} \langle \cdot, \vec{T} \rangle \dd \mu_T
\]
be an integral representation of $T=[\Gamma_{y,r}]$ as in \eqref{eq:integral}.

For any $\omega \in C^\infty_0(\wedge^n \R^{kn})$,
by $\Omega_{y,t}\Subset\Omega$
and the inner regularity of the Radon measure $\mu_T$, we have
\begin{multline*}
\bigl|(([\Gamma_{y,r}] \llcorner B^{kn}(y, t))\llcorner \psi_i)(\omega)
-([\Gamma_{y,r}] \llcorner B^{kn}(y, t))(\omega)\bigr|\\
\le|\langle\omega,\vec{T}\rangle|\mu_T(\Omega_{y,t}\setminus g^{-1}(A_i))
\to 0
\end{multline*}
as $i\to\infty$.

Since $J_{f_j}\in L^1_{\loc}(\Omega)$ and $\Omega_{y,t}\Subset\Omega$, we have 
for every $\omega \in C^\infty_0(\wedge^n \R^{kn})$,
by \eqref{eq:pointwise} and Lemma \ref{lem:area_pontwise_estimate}, that
\begin{align*}
  |([\Gamma_{y,r}]\llcorner\psi_i)(\omega) - [\Gamma_{y,t}](\omega)|
&=\biggl|\int_{\Omega_{y,t} \setminus g^{-1} A_i} g^*\omega\biggr|\\
&\leq \lambda_\omega n^{\frac{n}{2}}k^{\frac{n}{2}-1}K 
 \sum_{j=1}^k\int_{\Omega_{y,t} \setminus g^{-1} A_i}J_{f_j}\dd\Hausd^n\to 0
\end{align*}
as $i\to\infty$.

Having these estimates at our disposal, we conclude that, for each $\omega \in C^\infty_0(\wedge^n \R^{kn})$, we have
\begin{align*}
\left( [\Gamma_{y,r}]\llcorner B^{nk}(y,t)\right)(\omega) 
&= \lim_{i \to \infty} (([\Gamma_{y,r}] \llcorner B^{kn}(y, t))\llcorner \psi_i)(\omega) \\
&= \lim_{i \to \infty} ([\Gamma_{y,r}] \llcorner \psi_i)(\omega)
= [\Gamma_{y,t}](\omega).
\end{align*}
This completes the proof.
\end{proof}

\subsubsection{Slicing and isoperimetric estimates for $[\Gamma_{y,r}]$}

The first step towards the lower Ahlfors bound is the following slicing estimate for $[\Gamma_{y,r}]$ -- this is one of the key estimates in the proof of the lower Ahlfors bound.

\begin{lemma}
\label{lemma:integral_estimate_preliminary}
Let $\Omega\subset\bR^n$ be an open subset for $n \geq 2$,
$k\in\bN$, and let $f_1, \ldots, f_k \colon \Omega \to \R^n$ 
be non-constant $K$-quasiregular maps for some $K\ge 1$
such that $\max_{j\in\{1,\ldots,k\}}N(f_j) < \infty$. 
Let $g = (f_1, \ldots, f_k) \colon \Omega \to \R^{kn}$ and 
$\Gamma=g(\Omega)\subset\R^{kn}$.
Let $y\in \Gamma$ and $r>0$ be such that $g^{-1}(\Gamma_{y,r})\Subset\Omega$. Then, for every $t\in(0,r)$,
\begin{gather*}\label{eq:integral_estimate_preliminary}
\mass_{\R^{kn}} ([\Gamma_{y,t}])
\ge\frac{1}{n} \int_0^{t}
\mass_{\bR^{kn}} (\partial_{\bR^{kn}} [\Gamma_{y,s}]) \dd s.
\end{gather*}
\end{lemma}

\begin{proof}
Let $t\in (0,r)$. By Lemma \ref{lem:restriction_formula} and \eqref{eq:mass_restriction_interplay},
we have that 
\[
\mass_{\R^{kn}} ([\Gamma_{y,t}])=\mass_{B^{kn}(y,t)} ([\Gamma_{y,r}]).
\]
Similarly, by Lemmas \ref{lem:local_normality} and \ref{lem:well_defined_current}, we have
\[
\mass_{B^{kn}(y,r)} ([\Gamma_{y,r}]) + \mass_{B^{kn}(y,r)}(\partial_{B^{kn}(y,r)}[\Gamma_{y,r}]) = \mass_{B^{kn}(y,r)} ([\Gamma_{y,r}]) < \infty.
\]

Let now $h_y \colon \R^{nk} \to \R$ be the $1$-Lipschitz function $x\mapsto |x-y|$. Then $h_y^{-1}(-\infty,t) = B^{nk}(y,t)$ and, by Proposition \ref{lem:properties_of_slices}, we have
\[
\mass_{B^{kn}(y,t)} ([\Gamma_{y,r}]) \ge \frac{1}{n}\int_0^t \mass_{\R^{kn}}( \langle [\Gamma_{y,r}], h_y, s-\rangle ) \dd s.
\]
Since
\begin{align*}
\langle [\Gamma_{y,r}], h_y, s-\rangle 
&= \partial_{B^{kn}(y,r)}\left([\Gamma_{y,r}]\llcorner B^{kn}(y,s)\right) - (\partial_{B^{kn}(y,r)}[\Gamma_{y,r}])\llcorner B^{kn}(y,s) \\
&= \partial_{B^{kn}(y,r)}[\Gamma_{y,s}] - (\partial_{B^{kn}(y,r)}[\Gamma_{y,r}])\llcorner B^{kn}(y,s) 
= \partial_{B^{kn}(y,r)}[\Gamma_{y,s}]
\end{align*}
and $\partial_{B^{kn}(y,r)}[\Gamma_{y,s}] = \partial_{\R^{kn}}[\Gamma_{y,s}]$ for all $0<s\le t < r$, we have
\[
\mass_{\R^{kn}} ([\Gamma_{y,t}]) \ge \frac{1}{n}\int_0^t \mass_{\R^{kn}}( \partial_{\R^{kn}}[\Gamma_{y,s}] ) \dd s
\]
as claimed.
\end{proof}

We finish this section with an isoperimetric estimate for the currents $[\Gamma_{y,r}]$ -- this is the other key estimate in the proof of the lower Ahlfors bound.

\begin{lemma}\label{lem:push_isoperimetric_estimate}
Let $\Omega\subset\bR^n$ be an open subset for $n \geq 2$,
$k\in\bN$, and let $f_1, \ldots, f_k \colon \Omega \to \R^n$ 
be non-constant $K$-quasiregular maps for some $K\ge 1$
such that $\max_{j\in\{1,\ldots,k\}}N(f_j) < \infty$. 
Let $g = (f_1, \ldots, f_k) \colon \Omega \to \R^{kn}$ and 
$\Gamma=g(\Omega)\subset\R^{kn}$.
Let $y\in \Gamma$ and $r>0$ be such that $g^{-1}(\Gamma_{y,r}) \Subset\Omega$. Then there is a constant $C=C(n)>0$ depending only on $n$ such that, for Lebesgue almost every $t \in (0,r)$, we have
\begin{align*}
\left(\mass_{\R^{kn}} (\partial_{\R^{kn}}[\Gamma_{y,t}])\right)^{\frac{n}{n-1}}
\ge\frac{\mass_{\R^{kn}}([\Gamma_{y,t}])}{C(n)k^{\frac{n}{2}} K\min_j N(f_j)}.
\end{align*}
\end{lemma}

\begin{proof}
For every $t\in (0,r)$, by Lemmas \ref{lem:well_defined_current}, \ref{lem:area_thm}, and \ref{lem:area_pontwise_estimate}, we have
\begin{gather*}
\mass_{\R^{kn}} ([\Gamma_{y,t}]) 
\leq n^{\frac{n}{2}}K k^{\frac{n}{2}-1} \Bigl( \min_{j\in\{1,\ldots,k\}} N(f_j) \Bigr) \sum_{j=1}^k \int_{\Omega_{y,t}} J_{f_j}\dd\Hausd^n,
\end{gather*}
where $\Omega_{y,t} = g^{-1}(\Gamma_{y,t})$.

Let $\psi_t\in C_0^\infty(\wedge^n \R^n)$ be a function satisfying $0 \le \psi_t \le 1$ and $\omega_t|\Omega_{y,t} = \cH^n|\Omega_{y,t}$ as measures, where $\omega_t = \psi_t \vol_{\R^n}$.
Let also $j\in \{1,\ldots, k\}$. Then 
\begin{align*}
\int_{\Omega_{y,t}}J_{f_j}\dd\cH^n
&\le \int_{\Omega_{y,t}} \psi_t(f_j(x))J_{f_j}(x) \dd\cH^n(x)
= \int_{\Omega_{y,t}} f_j^* \omega_t 
= \int_{\Omega_{y,t}} g^* \proj_j^* \omega_t\\
&=[\Gamma_{y,t}](\proj_j^* \omega_t)
=((\proj_j)_*[\Gamma_{y,t}])(\omega_t)
\le\mass_{\R^n} ((\proj_j)_*[\Gamma_{y,t}]).
\end{align*}
Thus
\[
\mass_{\R^{kn}} ([\Gamma_{y,t}]) \le n^{\frac{n}{2}}K k^{\frac{n}{2}-1} \Bigl( \min_{j\in\{1,\ldots,k\}} N(f_j) \Bigr)\sum_{j=1}^n \mass_{\R^{n}}((\proj_j)_*[\Gamma_{y,t}]).
\]
Since 
\[
\mass_{\R^n} (\partial_{\R^n}((\proj_j)_*[\Gamma_{y,t}]))
= \mass_{\R^n} ((\proj_j)_*\partial_{\R^n} [\Gamma_{y,t}])
\leq \mass_{\R^{kn}}(\partial_{\R^n} [\Gamma_{y,t}]),
\]
it suffices to, for almost every $t\in (0,r)$, verify the isoperimetric inequality
\begin{equation}
\label{eq:mass_isoperimetric_est} 
\mass_{\R^n} ((\proj_j)_*[\Gamma_{y,t}])
 \leq C(n)\left(\mass_{\R^n} (\partial_{\R^n}((\proj_j)_*[\Gamma_{y,t}]))\right)^\frac{n}{n-1}
\end{equation}
for each $j\in \{1,\ldots, k\}$. We show that $(\proj_j)_*[\Gamma_{y,t}]$ satisfies the assumptions for the isoperimetric inequality for $n$-currents in \cite[4.5.9(31)]{Federer}. More precisely, we show that $(\proj_j)_*[\Gamma_{y,t}]$ is locally normal and satisfies $(\proj_j)_*[\Gamma_{y,t}] = \mathcal{L}^n\llcorner g$, where $g \colon \R^n \to \Z$ is measurable and compactly supported and $\mathcal{L}^n$ is the Lebesgue measure in $\R^n$.

Let $j\in \{1,\ldots, k\}$. Since $\proj_j$ is $1$-Lipschitz, we have
\[
\mass_{\R^n} ((\proj_j)_*[\Gamma_{y,t}])\leq \mass_{\R^{kn}}([\Gamma_{y,r}]) < \infty.
\]
By Lemma \ref{lemma:integral_estimate_preliminary}, we also have that $\mass_{\R^{kn}}([\Gamma_{y,t}]) < \infty$ for almost every $t\in (0,r)$. Thus
\[
\mass_{\R^n} (\partial_{\R^n}((\proj_j)_*[\Gamma_{y,t}]))
= \mass_{\R^n} ((\proj_j)_*\partial_{\R^n} [\Gamma_{y,t}])
\leq \mass_{\R^{kn}}(\partial_{\R^n} [\Gamma_{y,t}]) <\infty.
\]
Hence $(\proj_j)_*[\Gamma_{y,t}]$ is a normal current for almost every $t\in (0,r)$. 

Let $\omega \in C_0^\infty(\wedge^n \R^n)$. Then, by the change of variables,
\begin{align*}
((\proj_j)_*[\Gamma_{y,t}])(\omega)=[\Gamma_{y,t}](\proj_j^* \omega)
&= \int_{\Omega_{y,t}} g^* \proj_j^* \omega
= \int_{\Omega_{y,t}} f_j^* \omega\\
&= \int_{f_j(\Omega_{y,t})} N(f_j, z, \Omega_{y,t}) \omega(z)\\
&=\int_{\bR^n} N(f_j, z, \Omega_{y,t})\chi_{f_j(\Omega_{y,t})}\omega(z).
\end{align*}
Thus
\[
(\proj_j)_*[\Gamma_{y,t}] = \mathcal{L}^n\llcorner u_t,
\]
where $u_t\colon \R^n \to \N$ is the function $z\mapsto N(f_j, z, \Omega_{y,t})\chi_{f_j(\Omega_{y,t})}$. Since $u_t$ has compact support, we conclude that, by the isoperimetric inequality for $n$-currents \cite[4.5.9(31)]{Federer},
there exists $C=C(n)>0$, depending only on $n$, for which \eqref{eq:mass_isoperimetric_est} holds. The claim follows.
\end{proof}

\subsubsection{Proof of Proposition \ref{prop:lower_Ahlfors_Euclidean}}

The final ingredient in obtaining the proof of Proposition \ref{prop:lower_Ahlfors_Euclidean} is a variant of the Bihari--LaSalle inequality \cite{Bihari}, which in turn is a nonlinear generalization of Gr\"onwall's inequality. 

\begin{lemma}\label{lem:integral_inequality}
Let $n>1$ be an integer, $a > 0$, and $C>0$. 
Let also $g\in L_{\loc}^{(n-1)/n}([0,a])$ be a function for which $g>0$
Lebesgue almost everywhere on $(0,a)$ and 
\[
g(t)\ge C\int_0^t g^{\frac{n-1}{n}}(s) \dd s
\]
for almost every $t\in(0,a)$.
Then 
\[
g(t) \geq \left(\frac{C}{n}\right)^n t^n
\]
for almost every $t \in (0, a)$.
\end{lemma}
\begin{proof} 
Let $G\colon [0,a] \to \R$ be the function
\[
t \mapsto C \int_0^t g^{\frac{n-1}{n}}(s) \dd s.
\]
Then $G$ is absolutely continuous, non-decreasing on $[0,a]$, and positive on $(0,a)$. Thus,
 \[
 (G^{1/n})' 
 = \frac{G'}{nG^{\frac{n-1}{n}}}
 = \frac{C g^{\frac{n-1}{n}}}{nG^{\frac{n-1}{n}}}
 \geq \frac{C}{n}
\]
almost everywhere on $[0,a]$.
Since $G(0)=0$, we have for almost every $t \in (0, a)$ that
 \[
 g^{\frac{1}{n}}(t) \ge 
 G^{\frac{1}{n}}(t) - G^{\frac{1}{n}}(0) 
 \geq \int_0^t\bigl(G^{\frac{1}{n}}\bigr)'(s)\dd s
 \geq \int_0^t \frac{C}{n} \dd s
 = \frac{C}{n}t.
 \]
\end{proof}

\begin{proof}[Proof of Proposition \ref{prop:lower_Ahlfors_Euclidean}]
By Lemmas \ref{eq:integral_estimate_preliminary}, \ref{lem:push_isoperimetric_estimate} and \ref{lem:integral_inequality}, there exists a constant $C=C(n)>0$, depending only on $n$, for which 
\[
\mass_{\R^{kn}}([\Gamma_{y,r}]) \ge \biggl(\frac{1} {Ck^{\frac{n}{2}} K\min_j N(f_j)}\biggr)^{n-1} n^{-2 n} r^n.
\]
Since 
\[
 \Bigl( \min_{j\in\{1,\ldots,k\}}N(f_j) \Bigr)\cH^n(\Gamma_{y,r}) \ge \mass_{\R^{kn}}\left([\Gamma_{y,r}]\right)
\]
by Lemma \ref{lem:well_defined_current}, we conclude that
\begin{align*}
\cH^n(\Gamma_{y,r}) &\ge \left(\frac{1}{Ck^{\frac{n}{2}} K\min_j N(f_j)}\right)^{n-1} \frac{1}{n^{2 n}} \frac{1}{\min_j N(f_j)} r^n \\
&= \left( \frac{1}{C' k^\frac{n(n-1)}{2} K^{n-1} \left( \min_j N(f_j)\right)^n}\right) r^n,
\end{align*}
where $C=C(n)>0$ and $C'=C'(n)>0$ depend only on $n$. The proof is complete.
\end{proof}


\section{Proof of Theorem \ref{thm:Gromov_interpretation}}
\label{sec:Gromov_for_maps}

In this section, we prove Theorem \ref{thm:Gromov_interpretation} using Proposition \ref{prop:Gromov_interpretation_Euclidean}. We use the same notation as before. Given a Riemannian $n$-manifold $N$, $k\in \N$, and a subset $\Gamma \subset N^k$, we denote
\[
\Gamma_{y,r} = B_{N^k}(y,r) \cap \Gamma
\]
for $y\in N^k$ and $r>0$.

Our first goal is to prove a small scale version of Theorem \ref{thm:Gromov_interpretation}.

\begin{lemma}\label{lem:manifold_small_scale}
Let $M$ and $N$ be closed, connected, oriented Riemannian $n$-manifolds, and let $f_1, \ldots, f_k \colon M \to N$ be non-constant $K$-quasiregular maps $M \to N$. Let also $g =(f_1,\ldots, f_k) \colon M \to N^k$ and $\Gamma = g(M)$. Then there exists $\lambda>0$ depending only on $N$ and $f_1$ and having the property that, for all $y \in \Gamma$ and $r\in (0,\lambda/4)$, we have 
\begin{align*}
\frac{1}{\left(C(n) k^{\frac{n}{2}} K\right)^{n-1} \left(\min_j \deg f_j\right)^{n}} \le
\frac{\cH^n(\Gamma_{y,r})}{r^n} \leq C(n) k^\frac{n}{2} K \left(\max_j \deg f_j\right).
\end{align*}
\end{lemma}

\begin{proof}
Let $\cM$ be a finite cover of $M$ by smooth 2-bilipschitz charts $(U, \varphi)$ of $M$. For each $x\in M$, there exists a radius $r_x>0$ having the property that, for each $r\in (0,r_x)$, $U(f_1,x,r_x)$ is a normal neighborhood of $x$ with respect to $f_1$ satisfying $f(U(f_1,x,r)) = B_N(f(x),r)$. Thus there exists a finite cover $\cN$ of $N$ by smooth 2-bilipschitz charts $(V, \psi)$ with the property that, for each $(V,\psi)\in \cN$, each component of $f_1^{-1}V$ is contained in an element of $\cM$.

Let $\lambda > 0$ be a Lebesgue number of $\cN$, that is, for every $y \in N$, we have $B_n(y, \lambda) \subset V$ for some $(V, \psi) \in \cN$. Note that $\lambda$ depends only on the first map $f_1$, and neither on $k$ nor the remaining maps $f_j$. 

Let $x \in M$, $y = g(x)$, and $0 < r < \lambda/4$. We first consider the cube of balls $Q_\lambda = B_N(f_1(x), \lambda) \times \ldots \times B_N(f_k(x), \lambda)$. Then $B_{N^k}(y, r) \subset Q_\lambda$ and, for every $j \in \{1, \ldots, k\}$, we may fix a chart $(V_j, \psi_j) \in \cN$ for which $B_N(f_j(x), \lambda) \subset V_j$. Let also $\sigma = \psi_1 \times \ldots \times \psi_k \colon Q_\lambda \to \R^{kn}$ be a $2$-bilipschitz embedding.
	
We note that $g^{-1} Q_\lambda \subset f_1^{-1} V_1$. Since every component of $f_1^{-1} V_1$ is contained in a chart of $\cM$, there exists a partition $\{W_i\}_{i\in I}$ of $g^{-1}Q_\lambda$ into open sets $W_i \subset U_i$, where $(U_i,\varphi_i)\in \cM$ for each $i\in I$. Since we may further assume that the images of $\varphi_i \colon U_i \to \R^n$ are mutually disjoint, the map $\varphi \colon g^{-1}Q_\lambda \to \R^n$, defined by $\varphi|W_i = \varphi_i|W_i$ for each open set $W_i$, is a locally $2$-bilipschitz embedding.

We set now $\Omega = \varphi(g^{-1} Q_\lambda)$ and let $g' =(f'_1,\ldots, f'_k) \colon \Omega \to \R^{kn}$ be the map $g' = \sigma \circ g \circ \varphi^{-1}$. Then $f'_j = \psi_j \circ f_j \circ \varphi^{-1}$ for each $j\in \{1,\ldots, k\}$. Since $\varphi^{-1}$ and each $\psi_j$ is locally 2-bilipschitz, the maps $f'_j$ are $2^{4n} K$-quasiregular. 
	
	We are therefore in position to apply Proposition \ref{prop:Gromov_interpretation_Euclidean} on $g'$. We denote $\Gamma'_{y, t} = \sigma(\Gamma \cap Q_\lambda) \cap B^{kn}(\sigma(y), t)$ for $t > 0$, and obtain a constant $C=C(n)>0$ depending only on $n$ for which 

\begin{align*}
\cH^n\left(\Gamma'_{y, t} \right) \leq C(n) k^\frac{n}{2} K \left(\max_i \deg f_i\right) t^n
\end{align*}
and
\[
\cH^n\left(\Gamma'_{y, t} \right) \geq \left(\frac{1}{\min_i \deg f_i}\right)^{n} \left(\frac{1}{C(n) k^\frac{n}{2} K}\right)^{n-1} t^n
\]
for each $t > 0$ satisfying $(g')^{-1}B^{kn}(\sigma(y), t) \Subset \Omega$.

Since $\sigma$ is a $2$-bilipschitz embedding, we have 
\[
B^{kn}(\sigma(y), r/2) \subset \sigma(B_{N^k}(y,r)) \subset B^{kn}(\sigma(y), 2r).
\]
Therefore,
\[
2^{-n} \cH^n(\Gamma'_{y,r/2}) \leq \cH^n(\Gamma_{y,r}) \leq 2^n \cH^n(\Gamma'_{y,2r}).
\]

It suffices now to show that $(g')^{-1} B^{kn}(\sigma(y), 2r)$ is compactly contained in $\varphi(U)$. For this, note first that $\sigma^{-1} B^{kn}(\sigma(y), 2r) \subset \overline{B}_{N^k}(y, 4r)$. Since $g^{-1} \overline{B}_{N^k}(y, 4r)$ is a closed subset of the closed manifold $M$, it is compact. Since $\overline{B}_{N^k}(y, 4r) \subset Q_\lambda$, we have that $g^{-1} \overline{B}_{N^k}(y, 4r) \subset g^{-1} Q_\lambda$. Thus $(g')^{-1} B^{kn}(\sigma(y), 2r)$ is contained in the compact subset $\varphi( g^{-1} \overline{B}_{N^k}(y, 4r))$ of $\Omega$. 
\end{proof}

\subsection{Large scale estimates}

In order to prove Theorem \ref{thm:Gromov_interpretation}, it remains to extend the estimate of Lemma \ref{lem:manifold_small_scale} to the radii $r$ satisfying $\lambda/4 \leq r \leq \diam \Gamma$.

The following lemma completes the proof of the Ahlfors lower bound in Theorem \ref{thm:Gromov_interpretation}.

\begin{lemma}\label{lem:manifold_large_scale}
Let $M$ and $N$ be closed, connected, oriented Riemannian $n$-manifolds, and let $f_1, \ldots, f_k \colon M \to N$ be non-constant $K$-quasiregular maps $M \to N$. Let also $g =(f_1,\ldots, f_k) \colon M \to N^k$ and $\Gamma = g(M)$. Then there exists a constant $C=C(n,f_1,M,N)>0$, depending only on $n$, $f_1$, $M$, and $N$, with the property that, for each $y \in \Gamma$ and all $r\in (0,\diam \Gamma)$, 
\[
\cH^n(\Gamma_{y,r}) \geq \frac{1}{\left(C k^{\frac{n}{2}} K\right)^{(n-1)} \left(\min_j \deg f_j\right)^{n}} r^n.
\]
\end{lemma}
\begin{proof}
Let $\lambda>0$ be as in Lemma \ref{lem:manifold_small_scale}. It suffices to consider radii $\lambda/4 \leq r \leq \diam \Gamma$. 

Since $\diam \Gamma \leq \diam N^k = k^{1/2} \diam N$, we have that $r/(k^{1/2} \diam N) \leq 1$ and $4r/\lambda \geq 1$. Now, by Lemma \ref{lem:manifold_small_scale}, there exist constants $C=C(n,\lambda)>0$ and $C'=C'(n,\lambda,\diam N)>0$ for which 
\begin{align*}
\cH^n\left(\Gamma_{y, r} \right) 
&\geq \cH^n\left(\Gamma_{y, \lambda/8} \right)\\
&\geq C(n, \lambda)^{-1} k^{-\frac{n(n-1)}{2}} K^{-(n-1)} \left(\min_i \deg f_i\right)^{-n} \\
&\geq C'(n, \lambda, \diam N)^{-1} k^{-\frac{n(n-1)}{2}} K^{-(n-1)} \left(\min_i \deg f_i\right)^{-n} r^n. 
\end{align*}
Hence, we have obtained the lower bound of Theorem \ref{thm:Gromov_interpretation}. Moreover, since $\lambda$ only depends on $f_1$ and the Riemannian metrics on $M$ and $N$, we have that $C'(n, \lambda, \diam N)$ only depends on $n$, $f_1$, $M$, and $N$, and not on $k$ or the other maps $f_i$.
\end{proof}

For the upper bound, a similar observation as in the proof of the lower bound yields
\begin{align*}
\cH^n\left(\Gamma_{y, r} \right) \leq \cH^n\left(\Gamma \right)
\leq \frac{4^n}{\lambda^n} \cH^n\left(\Gamma \right) r^n.
\end{align*}
Hence, the problem of the upper bound reduces to estimating the Hausdorff measure $\cH^n$ of the entire set $\Gamma$, and hence to a global counterpart of Lemma \ref{lem:upper_hausdorff_bound} on closed manifolds. We state this as follows.

\begin{lemma}\label{lem:global_manifold_hausdorff_bound}
Let $M$ and $N$ be closed, connected, oriented Riemannian $n$-manifolds, and let $f_1, \ldots, f_k \colon M \to N$ be non-constant $K$-quasiregular maps $M \to N$. Let also $g =(f_1,\ldots, f_k) \colon M \to N^k$ and $\Gamma = g(M)$.
Then there exists a constant $C=C(n)>0$, depending only on $n$, for which
\begin{equation}\label{eq:gloabl_manifold_hausdorff_bound}
\cH^n\left(\Gamma \right) \leq C k^\frac{n}{2} K \left(\max_j \deg f_j\right) \cH^n\left(N\right).
\end{equation}
\end{lemma}

The upper bound for the Hausdorff measure in Theorem \ref{thm:Gromov_interpretation} follows now almost immediately using Lemma \ref{lem:global_manifold_hausdorff_bound} and the same observation as in the proof of the lower bound. We record the final piece of the proof of Theorem \ref{thm:Gromov_interpretation}.

\begin{proof}[Proof of Theorem \ref{thm:Gromov_interpretation}]
By Lemma \ref{lem:manifold_large_scale}, it remains to show that, there exists a constant $C>0$ depending only on $n$, $M$, $N$, and $f_1$ for which
\begin{equation}
\label{eq:devnull}
\cH^n(\Gamma(y,r)) \le C k^{\frac{n}{2}}K \left(\max_j \deg f_j\right) r^n.
\end{equation}
Let $\lambda>0$ be as in Lemma \ref{lem:manifold_small_scale}.

We consider two cases. By Lemma \ref{lem:manifold_small_scale}, there exists a constant $C'=C'(n)>0$ depending only on $n$ for which \eqref{eq:devnull} holds with $C'$ for $r\in (0,\lambda/4)$.

Suppose now that $r\ge \lambda/4$. Then by Lemma \ref{lem:global_manifold_hausdorff_bound} there exists a constant $C''=C''(n)$, depending only on $n$, for which 
\begin{align*}
\cH^n\left(\Gamma_{y, r} \right)
&\leq \cH^n\left(\Gamma \right)
\leq \frac{4^n}{\lambda^n} \cH^n\left(\Gamma \right)\cdot r^n \\
&\le\frac{4^n}{\lambda^n} \cdot C'' \cdot\cH^n(N)\cdot 
k^{\frac{n}{2}}K\Bigl(\max_{j}\deg f_j\Bigr)\cdot r^n \\
&=C''' k^{\frac{n}{2}}K\Bigl(\max_{j}\deg f_j\Bigr)\cdot r^n,
\end{align*}
where the constant $C'''$ depends only on $n$, $\lambda$, and $N$. Since $\lambda$ depends only on $f_1$ and the Riemannian metrics on $M$ and $N$, it suffices to take the maximum of the obtained constants $C'$ and $C'''$.
This completes the proof of Theorem \ref{thm:Gromov_interpretation}.
\end{proof}

It remains to prove Lemma \ref{lem:global_manifold_hausdorff_bound}. Since we were unable to locate a suitable version of the area formula for continuous Sobolev maps between closed manifolds, we give a hands-on proof based on the area formula for Sobolev functions in charts. For this reason, we begin by recalling a version of the Vitali covering theorem.

\begin{thm}\label{thm:Vitali_on_mflds}
Let $M$ be a Riemannian $n$-manifold and, for every $x \in M$, let $r_x > 0$. Then there exists an at most countable collection of disjoint open balls $\cB = \{B_1, B_2, \ldots\}$ for which every ball $B_i=B_M(x_i,r_i)$ in the collection satisfies $r_i < r_{x_i}$ and the set $M \setminus \cup \cB$ has $\Hausd^n$-measure zero.
\end{thm}
\begin{proof}
	A version for closed balls follows from Federer \cite[Theorem 2.8.18 and Section 2.8.9]{Federer} (see also Heinonen \cite[Example 1.15 (c) and (f)]{Heinonen-book}). An open ball version follows since every small enough closed ball on $M$ has a boundary of measure zero.
\end{proof}

We are now ready for the proof of Lemma \ref{lem:global_manifold_hausdorff_bound}.
\begin{proof}[Proof of Lemma \ref{lem:global_manifold_hausdorff_bound}]
For each $x \in M$, let 
\[
r_x = \sup\{r > 0 : g(B_M(x, r)) \subset B_{N^k}(g(x), \lambda/4)\}.
\]
Since $g$ is continuous, we have $r_x > 0$ for every $x\in M$. Let $\cB$ be a countable family of balls as in the Vitali covering theorem \ref{thm:Vitali_on_mflds}.

Let $B \in \cB$. By the same construction as in Lemma \ref{lem:manifold_small_scale}, we obtain $2$-bilipschitz embeddings $\varphi \colon B \to \R^n$ and $\sigma = \psi_1 \times \cdots \times \psi_k \colon g(B) \to \R^{kn}$, where mappings $\psi_j$ are smooth $2$-bilipschitz charts on $N$. Let also again $g' = (f'_1,\ldots, f'_k) \colon \varphi(B) \to \R^{kn}$ be the map with $2^{4n}K$-quasiregular component functions $f_j' = \psi_j \circ f_j \circ \varphi^{-1}$ for $j\in \{1,\ldots, k\}$.

Hence, we may use Lemmas \ref{lem:area_thm} and \ref{lem:area_pontwise_estimate} to obtain a constant $C=C(n)>0$, depending only on $n$, for which

	\begin{align*}
		\cH^n(\sigma(g(B)))
		&\leq \int_{\sigma(g(B))} N(g', y', \varphi(B)) \dd\cH^n(y')\\
		&\leq C(n) k^{\frac{n}{2}-1} K \sum_{j=1}^{k } \int_{\varphi(B)} J_{f_j'}(x') \dd \cH^n(x').
	\end{align*}
	Since $\sigma$ is a 2-bilipschitz embedding, we have
	\[
		\cH^n(g(B)) \leq 2^n \cH^n(\sigma(g(B))).
	\]
	Moreover, we may also estimate
	\begin{align*}
		\int_{\varphi(B)} J_{f_j'}(x') \dd \cH^n(x')
		&= \int_{\varphi(B)} J_{\psi_j \circ f_j}(\varphi^{-1}(x')) J_{\varphi^{-1}}(x') \dd \cH^n(x')\\
		&= \int_B J_{\psi_j \circ f_j}(z) \dd \cH^n(z)
		= \int_B J_{\psi_j}(f_j(z)) J_{f_j}(z) \dd \cH^n(z)\\
		&\leq 2^n \int_B J_{f_j}(z) \dd \cH^n(z).
	\end{align*}
	Now, by combining these estimates for all $B \in \cB$ and absorbing the constants into $C(n)$, we obtain
	\begin{align*}
		\cH^n(g(\cup \cB)) 
		&\leq C(n) k^{\frac{n}{2}-1} K \sum_{j=1}^{k } \int_{\cup\cB} J_{f_j} \dd\cH^n\\
		&\leq C(n) k^{\frac{n}{2}} K \left(\max_j \int_{M} J_{f_j} \dd\cH^n \right)\\
		&= C(n) k^{\frac{n}{2}} K \left(\max_j \deg f_j \right) \cH^n(N).
	\end{align*}
Finally, since $g$ satisfies the Lusin condition, we have that $g(\cup \cB)$ has full $\cH^n$-measure in $\Gamma$, and the claim follows.
\end{proof}



\section{The entropy upper bound: Proof of Theorem \ref{thm:main_result}}
\label{sec:Gromov-proof}

In this section, we conclude the proof of the entropy equality $h(f) = \log \deg f$. 
We give first the entropy upper bound in the case of quasiregular self-maps and then finish the proof of Theorem \ref{thm:main_result}. The argument is otherwise the same as in \cite[Chapter 5]{Gromov-2003}. 

In the following theorem, we use the notation $K(f)$ for the smallest distortion constant of the quasiregular map $f\colon M\to M$.

\begin{thm}\label{thm:Gromov-original}
Let $f \colon M \to M$ be a $K$-quasiregular self-map on a closed, oriented, and Riemannian $n$-manifold $M$. Then
\[
h(f) \le \log \deg f + n \cdot \limsup_{k \to \infty} \frac{\log K(f^k)}{k} \le \log \deg f + n\log K.
\]
\end{thm}

\begin{proof}
Let $M$ be a closed, connected, and oriented Riemannian $n$-manifold, 
$n\ge 2$, $K\ge 1$, and let $f\colon M\to M$ be a non-constant $K$-quasiregular self-map. Recall that, by Theorem \ref{lem:lov-lodn-bound}, 
\[
h(f) = h(\Gamma_{(\id_M,f)}) \le \lov(\Gamma_{(\id_M,f)}) - \lodn(\Gamma_{(\id_M,f)}),
\]
where $\Gamma_{(\id_M,f)}=(\id_M,f)(M)\subset M^2$ is the graph of $f$. For each $k\in\bN$, let $g_k := (\id_M, f,f^2,\ldots, f^{k}):M \to M^{k+1}$ and 
\begin{gather*}
\Gamma_{g_k}:=g_k(M) = \chain_k(\Gamma_{(\id_M,f)}). 
\end{gather*}

By Theorem \ref{thm:Gromov_interpretation}, there exists $C=C(n)>0$, depending only on $n$, such that, for each $y\in\chain_k(\Gamma_{(\id_M,f)})$ and $\eps\in(0,\diam M)$, we have
\begin{align*}
\Hausd^n\bigl(\chain_k(\Gamma_{(\id_M,f)}) \cap D_{k,\infty}(y, \eps)\bigr) &\geq 
\Hausd^n\bigl(\chain_k(\Gamma_{(\id_M,f)}) \cap B_{M^{k+1}}(y, \eps)\bigr) \\
&\geq \frac{\eps^n}{C\cdot(k+1)^\frac{n^2}{2} (K(f^k))^{n-1}}.
\end{align*}
Thus 
\begin{equation}
\label{eq:G-lodn}
\begin{split}
- \lodn(\Gamma_{(\id_M,f)})
&\leq \liminf_{\eps \to 0} \limsup_{k \to \infty} \frac{\log \bigl( C\cdot(k+1)^\frac{n^2}{2} (K(f^k))^{n-1} \eps^{-n} \bigr)}{k} \\
&= \liminf_{\eps \to 0} \limsup_{k \to \infty} \left( \frac{n-1}{k}\log K(f^k)\right) \\
&= (n-1) \limsup_{k \to \infty} \frac{\log K(f^k)}{k}.
\end{split}
\end{equation}

On the other hand, we have either by Theorem \ref{thm:Gromov_interpretation} or by Lemma \ref{lem:global_manifold_hausdorff_bound}, that
\[
\Hausd^n(\chain_k(\Gamma_{(\id_M,f)})) \leq C\cdot (k+1)^n K(f^k) (\deg f)^{k}\cdot (\diam M)^n.
\]
Thus
\begin{equation}
\label{eq:G-lov}
\begin{split}
\lov(\Gamma_{(\id_M,f)})
&\leq \limsup_{k \to \infty} \frac{1}{k} 
\log \bigl( C\cdot (k+1)^n K(f^k) (\deg f)^{k} (\diam M)^n \bigr)\\
&= \log \deg f + \limsup_{k \to \infty} \frac{\log K(f^k)}{k}.
\end{split}
\end{equation}

Combining the estimates \eqref{eq:G-lodn} and \eqref{eq:G-lov}, we obtain the upper bound
\begin{gather*}
h(f) \leq \lov(\Gamma_{(\id_M,f)}) - \lodn(\Gamma_{(\id_M,f)})
\le \log \deg f +n \frac{\log(K(f^k))}{k}.
\end{gather*}
Since $K(f^k) \le K^k$, the proof is complete.
\end{proof}

\begin{proof}[Proof of Theorem \ref{thm:main_result}]
The lower bound $h(f) \ge \log \deg f$ follows from the variational principle and the lower bound $h_{\mu_f}(f) \ge \log \deg f$ in Proposition \ref{prop:entropy_lower_bound} for the invariant measure $\mu_f$. Thus it remains to prove the upper bound using the variant of Gromov's argument we discussed in the previous section.
Since $K(f^k) \le K$ for each $k\in \N$, the upper bound $h(f)\le \log \deg f$ follows immediately from Theorem \ref{thm:Gromov-original}.
\end{proof}

\bibliographystyle{abbrv} 
\bibliography{KOPS} 

\begin{thebibliography}{10}

\bibitem{Bihari}
I.~Bihari.
\newblock {A generalization of a lemma of {B}ellman and its application to
  uniqueness problems of differential equations}.
\newblock {\em Acta Math. Acad. Sci. Hungar.}, 7:81--94, 1956.

\bibitem{Bojarski-Iwaniec}
B.~Bojarski and T.~Iwaniec.
\newblock {Analytical foundations of the theory of quasiconformal mappings in
  {${\bf R}^{n}$}}.
\newblock {\em Ann. Acad. Sci. Fenn. Ser. A I Math.}, 8(2):257--324, 1983.

\bibitem{Bonk-Heinonen_Acta}
M.~Bonk and J.~Heinonen.
\newblock {Quasiregular mappings and cohomology}.
\newblock {\em Acta Math.}, 186(2):219--238, 2001.

\bibitem{Bowen1971}
R.~Bowen.
\newblock {Entropy for group endomorphisms and homogeneous spaces}.
\newblock {\em Trans. Amer. Math. Soc.}, 153:401--414, 1971.

\bibitem{Favre-RivieraLetelier_ArchmiedeanDyn}
C.~Favre and J.~Rivera-Letelier.
\newblock Th{\`e}orie ergodique des fractions rationnelles sur un corps
  ultram{\'e}trique.
\newblock {\em Proc. Lond. Math. Soc.}, 100(1):116--154, 2010.

\bibitem{Federer}
H.~Federer.
\newblock {\em {Geometric measure theory}}.
\newblock {Die Grundlehren der mathematischen Wissenschaften, Band 153}.
  Springer-Verlag New York Inc., New York, 1969.

\bibitem{Goldstein-Troyanov_cohomology}
V.~Gol'dshtein and M.~Troyanov.
\newblock {A conformal de {R}ham complex}.
\newblock {\em J. Geom. Anal.}, 20(3):651--669, 2010.

\bibitem{Gromov-2003}
M.~Gromov.
\newblock {On the entropy of holomorphic maps}.
\newblock {\em Enseign. Math. (2)}, 49(3-4):217--235, 2003.

\bibitem{Haissinsky-Pilgrim_CoaConfDYn}
P.~Ha{\"i}ssinsky and K.~Pilgrim.
\newblock Coarse expanding conformal dynamics.
\newblock {\em Asterisque}, 325, 2009.

\bibitem{Hajlasz-areathm}
P.~Haj{\l}asz.
\newblock {Sobolev mappings, co-area formula and related topics}.
\newblock In {\em {Proceedings on {A}nalysis and {G}eometry ({R}ussian)
  ({N}ovosibirsk {A}kademgorodok, 1999)}}, pages 227--254. Izdat. Ross. Akad.
  Nauk Sib. Otd. Inst. Mat., Novosibirsk, 2000.

\bibitem{Heinonen-book}
J.~Heinonen.
\newblock {\em {Lectures on analysis on metric spaces}}.
\newblock {Universitext}. Springer-Verlag, New York, 2001.

\bibitem{HeinonenKilpelainenMartio2006book}
J.~Heinonen, T.~Kilpel{\"a}inen, and O.~Martio.
\newblock {\em {Nonlinear potential theory of degenerate elliptic equations}}.
\newblock Dover Publications, Inc., Mineola, NY, 2006.
\newblock Unabridged republication of the 1993 original.

\bibitem{Iwaniec-Martin_AASF}
T.~Iwaniec and G.~Martin.
\newblock {Quasiregular semigroups}.
\newblock {\em Ann. Acad. Sci. Fenn. Math.}, 21(2):241--254, 1996.

\bibitem{Iwaniec-Martin-book}
T.~Iwaniec and G.~Martin.
\newblock {\em {Geometric function theory and non-linear analysis}}.
\newblock {Oxford Mathematical Monographs}. The Clarendon Press, Oxford
  University Press, New York, 2001.

\bibitem{Kangaslampi-thesis}
R.~Kangaslampi.
\newblock {Uniformly quasiregular mappings on elliptic {R}iemannian manifolds}.
\newblock {\em Ann. Acad. Sci. Fenn. Math. Diss.}, 151, 2008.
\newblock 72pp.

\bibitem{Kangasniemi}
I.~{Kangasniemi}.
\newblock {Sharp cohomological bound for uniformly quasiregularly elliptic
  manifolds}.
\newblock {\em ArXiv e-prints}, Nov. 2017.

\bibitem{KangasniemiPankka}
I.~Kangasniemi and P.~Pankka.
\newblock {Uniform cohomological expansion of uniformly quasiregular mappings}.
\newblock {\em Proc. London Math. Soc.}, 118(3):701--728, 2019.

\bibitem{Katok-1977}
A.~B. Katok.
\newblock {The entropy conjecture}.
\newblock In {\em {Smooth dynamical systems ({R}ussian)}}, pages 181--203.
  Izdat. ``Mir'', Moscow, 1977.

\bibitem{Lyubich-1983}
M.~J. Ljubich.
\newblock {Entropy properties of rational endomorphisms of the {R}iemann
  sphere}.
\newblock {\em Ergodic Theory Dynam. Systems}, 3(3):351--385, 1983.

\bibitem{Mane-Sad-Sullivan}
R.~Ma{\~n}{\'e}, P.~Sad, and D.~Sullivan.
\newblock {On the dynamics of rational maps}.
\newblock {\em Ann. Sci. {\'E}cole Norm. Sup. (4)}, 16(2):193--217, 1983.

\bibitem{Martin-Mayer-Peltonen}
G.~Martin, V.~Mayer, and K.~Peltonen.
\newblock {The generalized {L}ichnerowicz problem: uniformly quasiregular
  mappings and space forms}.
\newblock {\em Proc. Amer. Math. Soc.}, 134(7):2091--2097 (electronic), 2006.

\bibitem{Martin-Peltonen-PAMS}
G.~Martin and K.~Peltonen.
\newblock {Sto\"\i low factorization for quasiregular mappings in all
  dimensions}.
\newblock {\em Proc. Amer. Math. Soc.}, 138(1):147--151, 2010.

\bibitem{Martin2014}
G.~J. Martin.
\newblock {The theory of quasiconformal mappings in higher dimensions, {I}}.
\newblock In {\em {Handbook of {T}eichm{\"u}ller theory. {V}ol. {IV}}},
  volume~19 of {\em {IRMA Lect. Math. Theor. Phys.}}, pages 619--677. Eur.
  Math. Soc., Z{\"u}rich, 2014.

\bibitem{Martin-Mayer-2003}
G.~J. Martin and V.~Mayer.
\newblock {Rigidity in holomorphic and quasiregular dynamics}.
\newblock {\em Trans. Amer. Math. Soc.}, 355(11):4349--4363, 2003.

\bibitem{Mayer1997paper}
V.~Mayer.
\newblock {Uniformly quasiregular mappings of {L}att{\`e}s type}.
\newblock {\em Conform. Geom. Dyn.}, 1:104--111, 1997.

\bibitem{Misiurewicz-Przytycki-1977}
M.~Misiurewicz and F.~Przytycki.
\newblock {Topological entropy and degree of smooth mappings}.
\newblock {\em Bull. Acad. Polon. Sci. S{\'e}r. Sci. Math. Astronom. Phys.},
  25(6):573--574, 1977.

\bibitem{OkuyamaPankka}
Y.~Okuyama and P.~Pankka.
\newblock {Equilibrium measures for uniformly quasiregular dynamics}.
\newblock {\em J. Lond. Math. Soc. (2)}, 89(2):524--538, 2014.

\bibitem{Peltonen-CGD}
K.~Peltonen.
\newblock {Examples of uniformly quasiregular mappings}.
\newblock {\em Conform. Geom. Dyn.}, 3:158--163, 1999.

\bibitem{Prywes}
E.~Prywes.
\newblock {A bound on the cohomology of quasiregularly elliptic manifolds}.
\newblock {\em Ann. of Math. (2)}, 189(3):863--883, 2019.

\bibitem{Przytycki-Urbanski_book}
F.~Przytycki and M.~Urba{\'n}ski.
\newblock {\em Conformal fractals: ergodic theory methods}.
\newblock Cambridge university press, 2010.

\bibitem{Rickman}
S.~Rickman.
\newblock {\em {Quasiregular mappings}}, volume~26 of {\em {Ergebnisse der
  Mathematik und ihrer Grenzgebiete (3) [Results in Mathematics and Related
  Areas (3)]}}.
\newblock Springer-Verlag, Berlin, 1993.

\bibitem{Rokhlin_measuretheory}
V.~A. Rohlin.
\newblock {On the fundamental ideas of measure theory}.
\newblock {\em Mat. Sbornik N.S.}, 25(67):107--150, 1949.

\bibitem{Rokhlin_entropy}
V.~A. Rohlin.
\newblock {Lectures on the entropy theory of transformations with invariant
  measure}.
\newblock {\em Uspehi Mat. Nauk}, 22(5 (137)):3--56, 1967.

\bibitem{Shub-BAMS-1974}
M.~Shub.
\newblock {Dynamical systems, filtrations and entropy}.
\newblock {\em Bull. Amer. Math. Soc.}, 80:27--41, 1974.

\bibitem{Srivastava_Borel}
S.~M. Srivastava.
\newblock {\em A course in Borel sets}.
\newblock Springer, 1998.

\bibitem{Vaisala1966paper}
J.~V{\"a}is{\"a}l{\"a}.
\newblock {Discrete open mappings on manifolds}.
\newblock {\em Ann. Acad. Sci. Fenn. A I}, 392:1--10, 1966.

\bibitem{Varopoulos-book}
N.~T. Varopoulos, L.~Saloff-Coste, and T.~Coulhon.
\newblock {\em {Analysis and geometry on groups}}, volume 100 of {\em
  {Cambridge Tracts in Mathematics}}.
\newblock Cambridge University Press, Cambridge, 1992.

\bibitem{Yomdin-1987}
Y.~Yomdin.
\newblock {Volume growth and entropy}.
\newblock {\em Israel J. Math.}, 57(3):285--300, 1987.

\end{thebibliography}

\end{document}